\documentclass[amstex,12pt, amssymb]{article}

\usepackage{mathtext}
\usepackage[cp1251]{inputenc}
\usepackage[T2A]{fontenc}
\usepackage[dvips]{graphicx}
\usepackage{amsmath}
\usepackage{amssymb}
\usepackage{amsxtra}
\usepackage{latexsym}
\usepackage{ifthen}

\newcounter{lemma}[section]

\newcounter{corollary}[section]

\newcounter{remark}[section]

\newcounter{theorem}[section]

\newcounter{proposition}[section]

\newcounter{example}

\numberwithin{equation}{section}

\begin{document}

\markboth{\centerline{E.~SEVOST'YANOV, O.~DOVHOPIATYI}}
{\centerline{ON COMPACT CLASSES...}}

\def\cc{\setcounter{equation}{0}
\setcounter{figure}{0}\setcounter{table}{0}}

\overfullrule=0pt


\author{{OLEKSANDR DOVHOPIATYI, EVGENY SEVOST'YANOV}\\}

\title{
{\bf ON COMPACT CLASSES OF SOLUTIONS OF THE DIRICHLET PROBLEM WITH
INTEGRAL RESTRICTIONS }}

\date{\today}
\maketitle

\begin{abstract}
This article is devoted to the study of the problem of compactness
of solutions of the differential Beltrami equation with
degeneration. We study the case when the complex characteristic of
the equations satisfies the constraints of integral type. In this
case, we have proved a theorem on the compactness of the class of
homeomorphic solutions of the Beltrami equation, which satisfy the
hydrodynamic normalization condition at infinity. Another important
result is the compactness theorem for the class of open discrete
solutions of the Dirichlet problem for the Beltrami equations in the
Jordan domain, the imaginary part of which vanishes at some
predetermined inner point.
\end{abstract}


\section{Introduction}
Recently, we have obtained a number of results on the equicontinuity
of families of mappings with direct and inverse Poletsky
inequalities, see, e.g.~\cite{Sev$_3$}, \cite{SevSkv$_2$} and
\cite{SSD}. The main goal of this article is to demonstrate how the
indicated results can be applied to compactness theorems for
solutions of the Beltrami equation and the Dirichlet problem for it.
Our focus is on two types of results. The first of them is the
compactness of the classes of solutions of Beltrami equations
without any boundary conditions. The second type of results relates
to the same equation, but with an additional requirement, when these
solutions have a continuous boundary extension and, moreover,
coincide on the boundary of a given domain with a fixed continuous
function.

\medskip
Concerning the application of convergence and compactness theorems
to differential equations, let us point out earlier publications
related to the Beltrami equations, see, e.g., \cite{Dyb$_1$},
\cite{GR$_1$}--\cite{GR$_2$}, \cite{KPRS}, \cite{LGR} and
\cite{RSY}. We especially mention the little-known article of
Dybov~\cite{Dyb$_1$}, published in an inaccessible journal. Here,
the compactness of the class of homeomorphic solutions of the
Dirichlet problem for the Beltrami equation in the unit disk is
proved. Note that the proof of Dybov's result essentially uses the
geometry of the unit disk, and therefore cannot be transferred to
the case of more general domains by analogy.

\medskip
In what follows, a mapping $f:D\rightarrow{\Bbb C}$ is assumed to be
{\it sense-preserving,} moreover, we assume that $f$ has partial
derivatives almost everywhere. Put $f_{\overline{z}} = \left(f_x +
if_y\right)/2$ and $f_z = \left(f_x - if_y\right)/2.$ The {\it
complex dilatation} of $f$ at $z\in D$ is defined as follows:
$\mu(z)=\mu_f(z)=f_{\overline{z}}/f_z$ for $f_z\ne 0$ and $\mu(z)=0$
otherwise. The {\it maximal dilatation} of $f$ at $z$ is the
following function:
\begin{equation}\label{eq1}
K_{\mu}(z)=K_{\mu_f}(z)=\quad\frac{1+|\mu (z)|}{1-|\mu\,(z)|}\,.
\end{equation}
Given a Lebesgue measurable function $\mu:D\rightarrow {\Bbb D},$
${\Bbb D}=\{z\in {\Bbb C}: |z|<1\},$ we define the {\it maximal
dilatation} of $f$ at $z$ the function $K_{\mu}(z)$ in~(\ref{eq1}).
Note that the Jacobian of $f$ at $z\in D$ is calculated by the
formula
$$J(z,
f)=|f_z|^2-|f_{\overline{z}}|^2\,.$$
It is easy to see that
$K_{\mu_f}(z)=\frac{|f_z|+|f_{\overline{z}}|}{|f_z|-|f_{\overline{z}}|}$
whenever partial derivatives of $f$ exist at $z\in D$ and, in
addition, $J(z, f)\ne 0.$

\medskip
We will call the {\it Beltrami equation} the differential equation
of the form
\begin{equation}\label{eq2}
f_{\overline{z}}=\mu(z)\cdot f_z\,,
\end{equation}
where $\mu=\mu(z)$ is a given function with $|\mu(z)|<1$ a.a. The
{\it regular solution} of~(\ref{eq2}) in the domain $D\subset{\Bbb
C}$ is a homeomorphism $f:D\rightarrow{\Bbb C}$ of the class $W_{\rm
loc}^{1, 1}(D)$ such that $J(z, f)\ne 0$ for almost all $z\in D.$

\medskip
In the extended Euclidean space $\overline{{{\Bbb R}}^n}={{\Bbb
R}}^n\cup\{\infty\},$ we use the so-called {\it chordal metric} $h$
defined by the equalities
\begin{equation}\label{eq3C}
h(x,y)=\frac{|x-y|}{\sqrt{1+{|x|}^2} \sqrt{1+{|y|}^2}}\,,\quad x\ne
\infty\ne y\,, \quad\,h(x,\infty)=\frac{1}{\sqrt{1+{|x|}^2}}\,,
\end{equation}
see e.g.~\cite[Definition~12.1]{Va}. For a given set
$E\subset\overline{{\Bbb R}^n},$ we set
\begin{equation}\label{eq47***}
h(E):=\sup\limits_{x,y\in E}h(x, y)\,.
\end{equation}
The quantity $h(E)$ in~(\ref{eq47***}) is called the {\it chordal
diameter} of the set $E.$ As usual, the family $\frak{F}$ of
mappings $f:D\rightarrow \overline{{\Bbb C}}$ is called {\it
normal,} if from each sequence $f_n\in \frak{F},$ $n=1,2,\ldots , $
one can choose a subsequence $f_{n_k},$ $k=1,2,\ldots ,$ converging
to some mapping $f:D\rightarrow \overline{{\Bbb C}}$ locally
uniformly with respect to the metric $h.$ If, in addition, $f\in
\frak{F},$ the family $\frak{F}$ is called {\it compact}.

\medskip
Let $K$ be a compact set in ${\Bbb C},$ $\mathcal{M}(\Omega)$ be a
function of the open set $\Omega,$ and $\Phi\colon\overline{{\Bbb
R}^{+}}\rightarrow\overline{{\Bbb R}^{+}}$ be a non-decreasing
function. Denote by $\frak{F}^{\mathcal{M}}_{\Phi}(K)$ the class of
all regular solutions $f:{\Bbb C}\rightarrow{\Bbb C}$ of the
equation~(\ref{eq2}) with complex coefficients $\mu$ equal to zero
outside~$K$ such that
\begin{equation}\label{eq1C}
f(z)=z+o(1)\quad \text{as}\quad z\rightarrow\infty\,,
\end{equation}
wherein
\begin{equation}\label{e3.3.1}
\int\limits_{\Omega}\Phi(K_{\mu}(z))\cdot\frac{dm(z)}{(1+|z|^2)^2}\leqslant
\mathcal{M}(\Omega)
\end{equation}
for each open set $\Omega\subset {\Bbb C}.$
The following statement is true.

\medskip
\begin{theorem}\label{th1}
{\sl\, Let $\Phi:\overline{{\Bbb R^{+}}}\rightarrow \overline{{\Bbb
R^{+}}}$ be an increasing convex function that satisfies the
condition
\begin{equation}\label{eq2A}
\int\limits_{\delta}^{\infty}\frac{d\tau}{\tau\Phi^{\,-1}(\tau)}=\infty
\end{equation}
for some~$\delta>\Phi(0).$ Suppose, in addition, that the function
$\mathcal{M}$ is bounded. Then the family
$\frak{F}^{\mathcal{M}}_{\Phi}(K)$ is compact in~${\Bbb C}.$}
\end{theorem}

\medskip
Let us now turn to the problem of the compactness of the classes of
solutions of the Dirichlet problem for Beltrami equation. Consider
the following Dirichlet problem:
\begin{equation}\label{eq2C}
f_{\overline{z}}=\mu(z)\cdot f_z\,,
\end{equation}
\begin{equation}\label{eq1A}
\lim\limits_{\zeta\rightarrow z}{\rm
Re\,}f(\zeta)=\varphi(z)\qquad\forall\,\, z\in \partial D\,,
\end{equation}
where $\varphi:\partial D\rightarrow {\Bbb R}$ is a given continuous
function. In what follows, we assume that $D$ is some simply
connected Jordan domain in ${\Bbb C}.$  A mapping $f:D\rightarrow
{\Bbb R}^n$ is called {\it discrete} if the preimage
$\{f^{\,-1}(y)\}$ of each point $y\,\in\,{\Bbb R}^n$ consist of
isolated points, and {\it open} if the image of any open set
$U\subset D $ is an open set in ${\Bbb R}^n.$ The solution of the
problem~(\ref{eq2C})--(\ref{eq1A}) will be called {\it regular} if
one of two is fulfilled: either $f(z)=const$ in $D,$ or $f$ is an
open discrete mapping of class $W_{\rm loc}^{1, 1}(D)$ such that
$J(z, f)\ne 0$ for almost all $z\in D.$

\medskip
Let us fix the point $z_0\in D$ and the function $\varphi.$ Let
$\frak{F}^{\mathcal{M}}_{\varphi, \Phi, z_0}(D)$ denotes the class
of all regular solutions $f:D\rightarrow{\Bbb C}$ of the Dirichlet
problem~(\ref{eq2C})--(\ref{eq1A}) that satisfies the condition
${\rm Im}\,f(z_0)=0$ and, in addition,
\begin{equation}\label{eq1D}
\int\limits_{\Omega}\Phi(K_{\mu}(z))\cdot\frac{dm(z)}{(1+|z|^2)^2}\leqslant
\mathcal{M}(\Omega)
\end{equation}
for any open set~$\Omega\subset D.$ The following statement
generalizes~\cite[Theorem~2]{Dyb$_1$} to the case of simply
connected Jordan domains.

\medskip
\begin{theorem}\label{th2A}
{\sl\, Let $D$ be some simply connected Jordan domain in ${\Bbb C},$
$\Phi:\overline{{\Bbb R^{+}}}\rightarrow \overline{{\Bbb R^{+}}}$ is
an increasing convex function such that the condition~(\ref{eq2A})
holds for some $\delta>\Phi (0).$ Assume that the function
$\mathcal{M}$ is bounded, and the function $\varphi$ in~(\ref{eq1A})
is continuous. Then the family $\frak{F}^{\mathcal{M}}_{\varphi,
\Phi, z_0}(D)$ is compact in $D.$}
\end{theorem}

\medskip
Further, the manuscript is structured as follows. In the following
sections, we give some auxiliary results are presented that concern
convergence of mappings, closure of classes, local and boundary
behavior. The proofs of the main results, Theorems~\ref{th1}
and~\ref{th2A}, are located in the last two sections.

\section{Ring homeomorphisms with constraints of integral type}

In what follows, $M$ denotes the $n$-modulus of a family of paths,
and the element $dm(x)$ corresponds to a Lebesgue measure in ${\Bbb
R}^n,$ $n\geqslant 2,$ see~\cite{Va}. Set
$$S(x_0, r)=\{x\in {\Bbb
R}^n: |x-x_0|=r\}\,,B(x_0, r)=\{x\in {\Bbb R}^n: |x-x_0|<r\}\,,$$
$${\Bbb B}^n:=B(0, 1)\,,\quad {\Bbb S}^{n-1}:=S(0, 1)\,,\quad \Omega_n=m({\Bbb B}^n)\,,\quad
\omega_{n-1}=\mathcal{H}^{n-1}({\Bbb S}^{n-1})\,,$$
where $\mathcal{H}^{n-1}$ means the $(n-1)$-dimensional Hausdorff
measure in ${\Bbb R}^n.$ In what follows, we set
$$h(A, B)=\inf\limits_{x\in A, y\in B}h(x, y)\,,$$
where $h$ is a chordal distance defined in~(\ref{eq3C}).
In addition, given domains $A, B\subset {\Bbb R}^n$ we put
$${\rm dist}\,(A, B)=
\inf\limits_{x\in A, y\in B}|x-y|\,.$$
Given sets $E$ and $F$ and a domain $D$ in $\overline{{\Bbb
R}^n}={\Bbb R}^n\cup \{\infty\},$ we denote by $\Gamma(E, F, D)$ the
family of all paths $\gamma:[0, 1]\rightarrow \overline{{\Bbb R}^n}$
joining $E$ and $F$ in $D,$ that is, $\gamma(0)\in E,$ $\gamma(1)\in
F$ and $\gamma(t)\in D$ for all $t\in (0, 1).$ Everywhere below,
unless otherwise stated, the boundary and the closure of a set are
understood in the sense of an extended Euclidean space
$\overline{{\Bbb R}^n}.$ Let $x_0\in\overline{D},$ $x_0\ne\infty,$
\begin{equation}\label{eq1F}
A=A(x_0, r_1, r_2)=\{ x\,\in\,{\Bbb R}^n : r_1<|x-x_0|<r_2\}\,.
\end{equation}
Let $Q:{\Bbb R}^n\rightarrow {\Bbb R}^n$ be a Lebesgue measurable
function satisfying the condition $Q(x)\equiv 0$ for $x\in{\Bbb
R}^n\setminus D.$ The mapping $f:D\rightarrow \overline{{\Bbb R}^n}$
is called a {\it ring $Q$-mapping at the point $x_0\in
\overline{D}\setminus \{\infty\}$}, if the condition
\begin{equation}\label{eq3*!gl0} M(f(\Gamma(S(x_0, r_1),\,S(x_0, r_2),\,A\cap D)))
\leqslant \int\limits_{A\cap D} Q(x)\cdot \eta^n(|x-x_0|)\, dm(x)
\end{equation}
holds for any $0<r_1<r_2<d_0:=\sup\limits_{x\in D}|x-x_0|$ and all
Lebesgue measurable functions $\eta:(r_1, r_2)\rightarrow [0,
\infty]$ such that
\begin{equation}\label{eq*3gl0}
\int\limits_{r_1}^{r_2}\eta(r)\,dr\geqslant 1\,.
\end{equation}

\medskip
The next important lemma was established earlier
in~\cite[Theorem~4.4]{RSS} for the case when the sequence of
functions $Q_m,$ $m=1,2,\ldots, $ considered in it is fixed.

\medskip
\begin{lemma}\label{lem1}
{\sl\, Let $D$ be a domain in ${\Bbb R}^n,$ $n\geqslant 2,$ and let
$f_k,$ $k=1,2,\ldots $ be a sequence of homeomorphisms of the domain
$D$ in ${\Bbb R}^n,$ which converges locally uniformly in $D$ to
some mapping $f: D\rightarrow \overline{{\Bbb R}^n}$ by chordal
metric $h.$ Suppose, moreover, that $\Phi:[0, \infty]\rightarrow [0,
\infty] $ is a strictly increasing convex function, and each map
$f_k,$ $k=1,2,\ldots $ satisfies the relation~(\ref{eq3*!gl0}) at
each point $x_0\in D$ with some function $Q=Q_k(x)$ such that
\begin{equation}\label{eq2!!}
\int\limits_D\Phi\left(Q_k(x)\right)\frac{dm(x)}{\left(1+|x|^2\right)^n}\
\leqslant\ M_0<\infty\,,\qquad k=1,2,\ldots\,.
\end{equation}
If
\begin{equation}\label{eq3!} \int\limits_{\delta_0}^{\infty}
\frac{d\tau}{\tau\left[\Phi^{-1}(\tau)\right]^{\frac{1}{n-1}}}=
\infty\,,
\end{equation}
for some $\delta_0>\tau_0:=\Phi(0),$ then $f$ is either a
homeomorphism $f:D\rightarrow {\Bbb R}^n,$ or a constant
$c\in\overline{{\Bbb R}^n}.$ }
\end{lemma}

\medskip
\begin{proof}
We use Lemma~4.1 in~\cite{RSS} (see also the estimates of the
integrals used in the proof of Theorem~1 in~\cite{Sev$_1$}). As
above, we put $A=A(x_0, r, R)=\{x\in{\Bbb R}^n: r<|x-x_0|<R\}.$ To
use~\cite[Lemma~4.1]{RSS} it is necessary to establish the existence
of sequences $0<r_m<R_m<\infty, $ $m=1,2,\ldots,$ such that
\begin{equation}\label{eq7}
M(f_k(\Gamma(S(x_0, r_m), S(x_0, R_m), A(x_0, r_m,
R_m))))\rightarrow 0
\end{equation}
as $m\rightarrow\infty$ uniformly over $k=1,2, \ldots .$ Choose an
arbitrary infinitely small sequence $R_m>0,$ $m=1,2, \ldots,$ and
fix a number $m\in{\Bbb N}.$
By~\cite[Lemma~7.3]{MRSY}
\begin{equation}\label{eq8}
M(f_k(\Gamma(S(x_0, r_m), S(x_0, R_m), A(x_0, r_m,
R_m))))\leqslant\frac{\omega_{n-1}}{I^{n-1}}\,,
\end{equation}
where
\begin{equation}\label{eq10}I=\int\limits_{r_m}^{R_m}\frac{dt}{tq^{\frac{1}{n-1}}_{k_{x_0}}(t)},\quad
q_{k_{x_0}}(t)=\frac{1}{\omega_{n-1}r^{n-1}} \int\limits_{S(x_0,
t)}Q_k(x)\,d\mathcal{H}^{n-1}\,.
\end{equation}
Using the substitution of variables $t=r/R_m, $ for any
$\varepsilon\in (0, R_m) $ we obtain that
\begin{equation}\label{eq34}
\int\limits_{\varepsilon}^{R_m}\frac{dr}{rq^{\frac{1}{n-1}}_{k_{x_0}}(r)}
=\int\limits_{\varepsilon/R_m}^1\frac{dt}{tq^{\frac{1}{n-1}}_{k_{x_0}}(tR_m)}
=\int\limits_{\varepsilon/R_m}^1\frac{dt}{t\widetilde{q}^{\frac{1}{n-1}}_{0}(t)}\,,
\end{equation}
where $\widetilde{q}_0(t)$ is the average integral value of the
function $\widetilde{Q}(x):=Q_k(R_mx+x_0)$ over the sphere $|x|=t,$
see the ratio~(\ref{eq10}). Then, according to~\cite[Lemma~3.1]{RS},
\begin{equation}\label{eq35}
\int\limits_{\varepsilon/R_m}^1\frac{dt}{t\widetilde{q}^{\frac{1}{n-1}}_{0}(t)}\geqslant
\frac{1}{n}\int\limits_{eM_*\left(\varepsilon/R_m\right)}^{\frac{M_*\left(\varepsilon/R_m\right)
R_m^n}{\varepsilon^n}}\frac{d\tau}
{\tau\left[\Phi^{-1}(\tau)\right]^{\frac{1}{n-1}}}\,,
\end{equation}
where
$$M_*\left(\varepsilon/R_m\right)=
\frac{1}{\Omega_n\left(1-\left(\varepsilon/R_m\right)^n\right)}
\int\limits_{A\left(0, \varepsilon/R_m, 1\right)} \Phi\left(Q_k(R_mx
+x_0)\right)\,dm(x)=$$
$$=
\frac{1}{\Omega_n\left(R_m^n-\varepsilon^n\right)}
\int\limits_{A\left(x_0, \varepsilon, R_m\right)}
\Phi\left(Q_k(x)\right)\,dm(x)\,.$$
Observe that $|x|\leqslant |x-x_0|+ |x_0|\leqslant R_m+|x_0|$ for
any~$x\in A(x_0, \varepsilon, R_m).$ Thus
$$M_*\left(\varepsilon/R_m\right)\leqslant \frac{\beta_m(x_0)}
{\Omega_n\left(R_m^n-\varepsilon^n\right)}\int\limits_{A(x_0,
\varepsilon, R_m)}
\Phi(Q_k(x))\frac{dm(x)}{\left(1+|x|^2\right)^n}\,,$$
where $\beta_m(x_0)=\left(1+(R_m+|x_0|)^2\right)^n.$ Therefore,
$$M_*\left(\varepsilon/R_m\right)\leqslant \frac{2\beta_m(x_0)}{\Omega_n R^n_m}M_0$$
for $\varepsilon\leqslant R_m /\sqrt[n]{2},$ where $M_0$ is a
constant in~(\ref{eq2!!}).
Observe that
$$M_*\left(\varepsilon/R_m\right)>\Phi(0)>0\,,$$
because $\Phi$ is increasing. Now, by~(\ref{eq34}) and~(\ref{eq35})
we obtain that
\begin{equation}\label{eq12}\int\limits_{\varepsilon}^{R_m}\frac{dr}{rq^{\frac{1}{n-1}}_{k_{x_0}}(r)} \geqslant
\frac{1}{n}\int\limits_{\frac{2\beta_m(x_0)M_0e}{\Omega_nR_m^n}}
^{\frac{\Phi(0)R_m^n}{\varepsilon^n}}\frac{d\tau}
{\tau\left[\Phi^{\,-1}(\tau)\right]^{\frac{1}{n-1}}}\,.
\end{equation}
From the conditions~(\ref{eq3!}) and~(\ref{eq12}) it follows that
there exists a number $0<r_m<R_m $ such that
\begin{equation}\label{eq13}
\int\limits_{r_m}^{R_m}\frac{dr}{rq^{\frac{1}{n-1}}_{k_{x_0}}(r)}\geqslant
2^m\,.
\end{equation}
Finally, it follows from~(\ref{eq8}) and~(\ref{eq13}) that there are
infinitesimal positive sequences $r_m$ and $R_m$ satisfying the
condition~(\ref{eq7}). Then by~\cite[Lema~4.1]{RSS} the map $f$ is
either a homeomorphism in ${\Bbb R}^n,$ or a constant in
$\overline{{\Bbb R}^n},$ which was required to prove.
\end{proof}~$\Box$

\medskip
According to~\cite{GM}, a domain $D\subset{\Bbb R}^n$ is called the
{\it quasiextremal distance domain}, abbr. {\it $QED$-domain}, if
there exists a number $A\geqslant 1 $ such that the inequality
\begin{equation}\label{eq4***}
M(\Gamma(E, F, {\Bbb R}^n))\leqslant A\cdot M(\Gamma(E, F, D))
\end{equation}
holds for any continua $E$ and $F$ in $D.$
The next assertion was established by the second co-author
in~\cite[Lemmas~3.1, 3.2]{Sev$_3$} for the case of a fixed function
$Q.$ However, it is fundamentally important for us to prove a
similar assertion when the functions $Q$ can change, but are subject
to condition~(\ref{eq2!!}).

\medskip
\begin{lemma}\label{lem2}
{\sl\, Let $D$ and $D^{\,\prime}$ be domains in ${\Bbb R}^n,$
$n\geqslant 2,$ $b_0\in D,$ $b_0^{\,\prime}\in D^{\,\prime},$ and
let $f_k:D\rightarrow D^{\,\prime},$ $k=1,2,\ldots, $ be a family of
homeomorphisms of~$D$ onto $D^{\,\prime}$ with
$f_k(b_0)=b^{\,\prime}_0,$ $k=1,2, \ldots. $ Suppose that any $f_k,$
$k=1,2,\ldots ,$ satisfies the relation~(\ref{eq3*!gl0}) at any
point $x_0\in \overline{D}$ and some function $Q=Q_k(x)\geqslant 1 $
such that
\begin{equation}\label{eq2!!A}
\int\limits_D\Phi\left(Q_k(x)\right)\frac{dm(x)}{\left(1+|x|^2\right)^n}\
\leqslant M_0<\infty\,,\qquad k=1,2,\ldots\,.
\end{equation}
Let $D$ be locally connected on the boundary, and let $D^{\,\prime}$
be a $QED$-domain containing at least one finite boundary point. If
\begin{equation}\label{eq3!A} \int\limits_{\delta_0}^{\infty}
\frac{d\tau}{\tau\left[\Phi^{-1}(\tau)\right]^{\frac{1}{n-1}}}=
\infty
\end{equation}
for some $\delta_0>\tau_0:=\Phi(0),$ then each $f_k,$
$k=1,2,\ldots,$ has a continuous extension
$f_k:\overline{D}\rightarrow\overline{D^{\,\prime}}$ and, besides
that, the family of all extended mappings $\overline{f}_k,$
$k=1,2,\ldots,$ is equicontinuous in $\overline{D}.$
 }
\end{lemma}

\medskip
\begin{proof}
The equicontinuity of ${\{f_k\}}^{\infty}_{k=1}$ at inner points of
$D$ follows by~\cite[Theorem~4.1]{RS}. It is only necessary to prove
the possibility of continuous boundary extension of each $f_k,$
$k=1,2, \ldots,$ , as well as the equicontinuity of the family of
extended mappings $\overline{f}_k$ at $\partial D.$ We may assume
that all functions $Q_k(x),$ $k=1,2, \ldots, $ are extended by the
rule $Q_k(x)\equiv 1$ for $x\in {\Bbb R}^n\setminus D.$ Let $x_0\in
\partial D,$ and let $0<\varepsilon_0<\delta(x_0):=\sup\limits_{x\in
D}|x-x_0|,$ $ 0<\varepsilon <\varepsilon_0.$ Consider the function
$I_k(\varepsilon, \varepsilon_0)=\int\limits
_{\varepsilon}^{\varepsilon_0}\psi_k(t)\,dt,$ where
\begin{equation}\label{eq1*****}\psi_k(t)\quad=\quad \left \{\begin{array}{rr}
1/[tq_{k_{x_0}}^{\frac{1}{n-1}}(t)]\ , & \ t\in (\varepsilon,
\varepsilon_0)\ ,
\\ 0\ ,  &  \ t\notin (\varepsilon,
\varepsilon_0)\,,
\end{array} \right.
\end{equation}
and $q_{k_{x_0}}$ is defined by~(\ref{eq10}). Note that
$I_k(\varepsilon, \varepsilon_0)<\infty$ for any $\varepsilon\in (0,
\varepsilon_0).$
Arguing similarly to the proof of the relation~(\ref{eq12}), we may
show that
\begin{equation}\label{eq12B}\int\limits_{a}^{b}
\frac{dr}{rq^{\frac{1}{n-1}}_{k_{x_0}}(r)} \geqslant
\frac{1}{n}\int\limits_{\frac{2\beta(x_0)M_0e}{\Omega_nb^n}}
^{\frac{\Phi(0)b^n}{a^n}}\frac{d\tau}
{\tau\left[\Phi^{\,-1}(\tau)\right]^{\frac{1}{n-1}}}\,,
\end{equation}
for every $0<b<\varepsilon_0 $ and sufficiently small $0<a<b,$ where
$\beta(x_0)=(1+(b+|x_0|)^2)^n.$ By~(\ref{eq3!A}) and (\ref{eq12B}),
$I_k(\varepsilon, \varepsilon_0)>0$ for $\varepsilon\in (0,
\varepsilon^{\,\prime}_0)$ and some
$0<\varepsilon^{\,\prime}_0<\varepsilon_0.$ Using direct
calculations and Fubini's theorem (see also~\cite[Lemma~7.4]{MRSY})
we obtain that
$$\int\limits_{\varepsilon<|x-x_0|<\varepsilon_0}
Q_k(x)\cdot\psi^n(|x-x_0|)\,dm(x)=\omega_{n-1}\cdot I_k(\varepsilon,
\varepsilon_0)\,.$$
Now, by~\cite[Lemma~3.1]{Sev$_3$} there exists
$\widetilde{\varepsilon_0}=\widetilde{\varepsilon_0}(x_0)\in (0,
\varepsilon_0)$ such that
\begin{equation}
\label{eq3.10} h(f_k(E_1))\leqslant \frac{\alpha_n}{\delta}\,
\exp\left\{-\beta I_k(\sigma, \varepsilon_0)\cdot
(\alpha(\sigma))^{-1/(n-1)}\right\}\,,\quad k=1,2,\ldots\,,
\end{equation}
for any $\sigma\in (0, \widetilde{\varepsilon_0})$ and for any
continuum $E_1\subset B(x_0, \sigma)\cap D,$ where $h(f_k(E_1)$ is
defined in~(\ref{eq47***}). Here we use the notation
\begin{equation}\label{eq1.3}
\alpha(\sigma)=\left(
1+\frac{\int\limits_{\widetilde{\varepsilon_0}}^{\varepsilon_0}\psi_k(t)\,dt}
{\int\limits_{\sigma}^{\widetilde{\varepsilon_0}}\psi_k(t)\,dt}\right)^n\,,\end{equation}
besides that, $\delta = \frac{1}{2}\cdot h\left(b_0^{\,\prime},
\partial D^{\,\prime}\right),$
$\alpha_n$ is some constant depending only on $n,$ $A$ is a constant
from the definition of $QED$-domain $D^{\,\prime}$ in~(\ref{eq4***})
and $\beta = \left(\frac{1}{A}\right)^{\frac{1}{n-1}}.$ Since
$q_{k_{x_0}}(t)\geqslant 1$ for almost any $t\in (0,
\varepsilon_0),$ we obtain that
\begin{equation}\label{eq1I}
\int\limits_{\widetilde{\varepsilon_0}}^{\varepsilon_0}\psi_k(t)\,dt\leqslant
\log\frac{\varepsilon_0}{\widetilde{\varepsilon_0}}:=C_1>0\,,\qquad
k=1,2,\ldots \,.
\end{equation}
Arguing similarly to the proof of relation~(\ref{eq12}), , and
taking into account relation~(\ref{eq3!A}), we may show that
\begin{equation}\label{eq12E}
\int\limits_{\sigma}^{\widetilde{\varepsilon_0}}\psi_k(t)\,dt
\geqslant
\frac{1}{n}\int\limits_{\frac{2C_2(x_0)M_0e}{\Omega_n{\widetilde{\varepsilon_0}}^n}}
^{\frac{\Phi(0){\widetilde{\varepsilon_0}}^n}{\sigma^n}}\frac{d\tau}
{\tau\left[\Phi^{\,-1}(\tau)\right]^{\frac{1}{n-1}}}\geqslant
1\,,\qquad k=1,2,\ldots\,,
\end{equation}
for some $\sigma_1\in(0, \widetilde {\varepsilon_0})$ and all
$0<\sigma <\sigma_1,$ where
$C_2(x_0)=(1+(\widetilde{\varepsilon_0}+|x_0|)^2)^n.$
By~(\ref{eq3.10}), (\ref{eq1I}) and~(\ref{eq12E}) we obtain that
\begin{equation}\label{eq3.11}
h(f_k(E_1))\leqslant \frac{\alpha_n}{\delta}\, \exp\left\{-\beta
\cdot C_3\cdot
\frac{1}{n}\int\limits_{\frac{2C_4(x_0)M_0e}{\Omega_n\varepsilon^n_0}}
^{\frac{\Phi(0)\varepsilon_0^n}{\sigma^n}}\frac{d\tau}
{\tau\left[\Phi^{\,-1}(\tau)\right]^{\frac{1}{n-1}}}\right\}\,,
\quad k=1,2,\ldots,
\end{equation}
for any $\sigma\in (0, \sigma_2)$ and some $0<\sigma_2<\sigma_1,$
where $C_3~:=~(1+C_1)^{-n/(n-1)}$ and
$C_4(x_0)~:=~(1+(\varepsilon_0+|x_0|)^2)^n>0$ is some constant.
Due to the condition~(\ref{eq3!A}), the relation~(\ref{eq3.11})
implies the existence of a non-negative function
$\Delta=\Delta(\sigma)$ such that
\begin{equation}\label{eq3.11A}
h(f_k(E_1))\leqslant \Delta(\sigma)\rightarrow 0\,,\quad
\sigma\rightarrow 0\,,\quad k=1,2,\ldots \,.
\end{equation}
Note that $QED$-domains have the so-called strongly accessible
boundaries (see, e.g., \cite[Remark~13.10]{MRSY}). Further
considerations are made similarly to the proof of Lemma~3.2
in~\cite{Sev$_3$}. Namely, the possibility of extension of $f_k$ to
a continuous mapping $\overline{f}_k:\overline{D}\rightarrow
\overline{D^{\,\prime}}$ follows from~\cite[Theorem~2]{Sev$_4$}. Let
us to prove the equicontinuity of
${\{\overline{f}_k\}}^{\infty}_{k=1}$ at $\partial D.$ Suppose the
contrary, namely, that the family of mappings
${\{\overline{f}_k\}}^{\infty}_{k=1}$ is not equicontinuous at some
point $x_0\in \partial D. $ We can consider that $x_0\ne \infty. $
Then there is a number $a> 0$ such that for each $m=1,2,\ldots$
there is $x_m\in\overline{D}$ and $\overline{f}_{k_m}$ such that
$|x_0-x_m|< 1/m$ and, simultaneously,
\begin{equation}\label{eq42***}
h(\overline{f}_{k_m}(x_m), \overline{f}_{k_m}(x_0))\geqslant a\,.
\end{equation}
Since $f_k$ have a continuous extension on $\partial D,$ we may
consider that $x_m\in D,$ $m=1,2,\ldots .$ Besides that, by
continuity of $\overline{f}_{k_m}$ at $x_0$ there exists a sequence
$x^{\,\prime}_m\in D$ such that
$$h(\overline{f}_{k_m}(x^{\,\prime}_m),
\overline{f}_{k_m}(x_0))\leqslant a/2\,.$$ Now, by~(\ref{eq42***})
and by the triangle inequality we obtain that
\begin{equation}\label{eq42A}
h(f_{k_m}(x_m), f_{k_m}(x^{\,\prime}_m))\geqslant a/2\,.
\end{equation}
Since $D$ is locally connected at $x_0,$ there is a sequence of open
neighborhoods $V_m$ of $x_0$ such that the sets $W_m:=D\cap V_m$ are
connected and $W_m \subset B(x_0, 2^{\,-m}).$ Turning to the
subsequence, if necessary, we may assume that $x_m,
x_m^{\,\prime}\in W_m.$ Put $0<\varepsilon_0<\sup\limits_{x\in
D}|x-x_0|.$ We may consider that $B(x_0, 2^{\,-m})\subset B(x_0,
\varepsilon_0)$ for any $m=1,2,\ldots .$ Since $W_m$ is open and
connected, it is also path connected (see, e.g.,
\cite[Proposition~13.1]{MRSY}). Thus, the points $x_m$ and
$x^{\,\prime}_m$ may be joined by a path $\gamma_m$ in $W_m.$ Set
$E_m:=|\gamma_m|,$ where, as usual, $|\gamma|=\{x\in {\Bbb
R}^n:\,\,\exists\,t\in[a, b]: \gamma(t)=x\}$ denotes the locus of
$\gamma.$  Now, by~(\ref{eq3.11A})
$$h(f_{k_m}(E_m))\leqslant \Delta(2^{\,-m})\rightarrow 0\,,\quad m\rightarrow\infty\,.$$
The last relation contradicts~(\ref{eq42A})), which completes the
proof.~$\Box$
\end{proof}

\section{Proof of Theorem~\ref{th1}}

{\bf I.} First of all, we prove that the family
$\frak{F}^{\mathcal{M}}_{\Phi}(K)$ is equicontinuous. Fix
$f\in\frak{F}^{\mathcal{M}}_{\Phi}(K),$ arbitrary compact set
$C\subset {\Bbb C}$ and put $\widetilde{f}=\frac{1}{f(1/z)}.$  Since
$f(z)=z+o(1)$ as $z\rightarrow\infty,$
$\lim\limits_{z\rightarrow\infty}f(z)=\infty.$ Set
$\widetilde{f}(0)=0.$ If $f(1/w_0)=0,$ we set
$\widetilde{f}(w_0):=\infty$. Since $f(z)=z+o(1)$ as
$z\rightarrow\infty,$ there is a neighborhood $U$ of the origin and
a function $\varepsilon: U\rightarrow {\Bbb C}$ such that
$f(1/z)=1/z+\varepsilon(1/z),$ where $z\in U$ and
$\varepsilon(1/z)\rightarrow 0$ as $z\rightarrow 0.$ Thus,
$$\frac{\widetilde{f}(\Delta z)-\widetilde{f}(0)}{\Delta z}=
\frac{1}{\Delta z}\cdot \frac{1}{1/(\Delta z)+\varepsilon(1/\Delta
z)}=\frac{1}{1+(\Delta z)\cdot \varepsilon(1/\Delta z)}\rightarrow 1
$$
as $\Delta z\rightarrow 0.$ This proves that there exists
$\widetilde{f}^{\,\prime}(0),$ and $\widetilde{f}^{\,\prime}(0)=1.$
Since $\mu(z)$ vanishes outside $K,$ the mapping $f$ is conformal in
some neighborhood $V:={\Bbb C}\setminus B(0, 1/r_0)$ of the
infinity, and the number $1/r_0$ depends only on $K,$ and $K\subset
B(0, 1/r_0).$ Without loss of generality, we may assume that the
compact set $C$ also satisfies the condition $C\subset B (0, 1 /
r_0).$ In this case, the mapping $\widetilde{f}=\frac{1}{f(1/z)}$ is
conformal in $B(0, r_0).$ In addition, the mapping
$F(z):=\frac{1}{r_0}\cdot \widetilde{f}(r_0z)$ is a homeomorphism of
the unit disk such that $F(0)=0$ and $F^{\,\prime}(1)=1.$ By Koebe's
theorem on~1/4 (see, e.g., \cite[Theorem~1.3]{CG},
cf.~\cite[Theorem~1.1.3]{GR$_2$}) $F({\Bbb D})\supset B (0, 1/4).$
Then
\begin{equation}\label{eq4}
\widetilde{f}(B(0, r_0))\supset B(0, r_0/4)\,.
\end{equation}
By~(\ref{eq4})
\begin{equation}\label{eq5}
(1/f)(\overline{\Bbb C}\setminus \overline{B(0, 1/r_0)})\supset B(0,
r_0/4)\,.
\end{equation}
Taking into account formula~(\ref{eq5}), we show that
\begin{equation}\label{eq6}
f(\overline{\Bbb C}\setminus \overline{B(0, 1/r_0)})\supset
\overline{{\Bbb C}}\setminus \overline{B(0, 4/r_0)}\,.
\end{equation}
Indeed, let $y\in \overline{{\Bbb C}}\setminus \overline{B(0,
4/r_0)}.$ Now, $\frac{1}{y}\in B(0, r_0/4).$ By~(\ref{eq5}),
$\frac{1}{y}=(1/f)(x),$ $x\in \overline{\Bbb C}\setminus
\overline{B(0, 1/r_0)}.$ Thus, $y=f(x),$ $x\in \overline{\Bbb
C}\setminus \overline{B(0, 1/r_0)},$ which proves~(\ref{eq6}).

\medskip
Since $f$ is a homeomorphism in~${\Bbb C},$ by~(\ref{eq6}) we obtain
that
$$f(B(0, 1/r_0))\subset B(0, 4/r_0)\,.$$
Set $\Delta:=h(\overline{\Bbb C}\setminus B(0, 4/r_0)),$ where
$h(\overline{\Bbb C}\setminus B(0, 4/r_0))$ is defined
in~(\ref{eq47***}) for $E=\overline{\Bbb C}\setminus B(0, 4/r_0).$
By~\cite[Theorem~3.1]{LSS} $f$ is a ring $Q$-homeomorphism in ${\Bbb
C}$ for $Q=K_{\mu}(z),$ where $\mu$ is defined in~(\ref{eq2}), and
$K_{\mu}$ is defined in~(\ref{eq1}). In this case,
$\frak{F}^{\mathcal{M}}_{\Phi}(K)$ is equicontinuous in~$B(0,
1/r_0)$ by~\cite[Theorem~4.1]{RS}. Let $f_n\in
\frak{F}^{\mathcal{M}}_{\Phi}(K),$ $n=1,2,\ldots .$ By the
Arzela-Ascoli theorem (see, e.g., \cite[Theorem~20.4]{Va}) there is
a subsequence $f_{n_k}(z)$ of $f_n,$ $k=1,2,\ldots ,$ and a
continuous mapping $f:B(0, 1/r_0)\rightarrow\overline{\Bbb C}$ such
that $f_{n_k}$ converges to $f$ in $B(0, 1/r_0)$ locally uniformly
as $k\rightarrow\infty.$ In particular, since $C$ belongs to $B(0,
1/r_0),$ a sequence $f_{n_k}$ converges to $f$ uniformly in $C.$
Since $C$ was chosen arbitrarily, we proved that the family of
mappings $f_{n_k}$ converges to the mapping $f$ locally uniformly.

\medskip
{\bf II.} To complete the proof of Theorem~\ref{th1}, it remains to
establish that $f\in \frak{F}^{\mathcal{M}}_{\Phi}(K).$ First of
all, we prove that the limit mapping $f$ satisfies the condition
$f(z)=z+o(1)$ as $z\rightarrow\infty.$ Note that the family mappings
$F_{n_k}(z):=\frac{1}{r_0}\cdot\frac{1}{f_{n_k}(\frac{1}{r_0z})}$ is
compact in the unit disk (see, e.g., \cite[Theorem~1.10]{CG},
cf.~\cite[Theorem~1.2 Ch.~I]{GR$_2$}). Without loss of generality,
we may consider that $F_{n_k}$ converges locally uniformly in ${\Bbb
D}.$ Now $F(z)=\frac{1}{r_0}\cdot\frac{1}{f(\frac{1}{r_0z})}$
belongs to the class $S,$ consisting of conformal mappings $F$ of
the unit disk that satisfy the conditions $F(0)=0,$
$F^{\,\prime}(0)=1.$ Then the expansions of functions $F$ and
$F_{n_k}$ in a Taylor series at the origin have the form
\begin{equation}\label{eq1E}F_{n_k}(z)=z+c_kz^2+z^2\cdot \varepsilon_k(z)\,,\quad
k=1,2,\ldots\,,
\end{equation}
\begin{equation}\label{eq2K}
F(z)=z+c_0z^2+z^2\cdot \varepsilon_0(z)\,,
\end{equation}
where $\varepsilon_k(z)$ and $\varepsilon_0(z)$ tend to zero as
$z\rightarrow 0.$ It follows from~(\ref{eq1E}) and (\ref{eq2K}) that
\begin{equation}\label{eq3A}
f_{n_k}(t)=\frac{r_0t^2}{r_0t+c_k+\varepsilon_k\left(\frac{1}{r_0t}\right)}\,,\quad
f_{n_k}(t)-t=-\frac{c_k+\varepsilon_k\left(\frac{1}{r_0t}\right)}
{r_0+\frac{c_k}{t}+\frac{\varepsilon_k\left(\frac{1}{r_0t}\right)}{t}}\,,
\quad k=1,2,\ldots\,,
\end{equation}
\begin{equation}\label{eq4B}
f(t)=\frac{r_0t^2}{r_0t+c_0+\varepsilon_0\left(\frac{1}{r_0t}\right)}\,,\quad
f(t)-t=-\frac{c_0+\varepsilon_0\left(\frac{1}{r_0t}\right)}
{r_0+\frac{c_0}{t}+\frac{\varepsilon_0\left(\frac{1}{r_0t}\right)}{t}}\,.
\end{equation}
In particular, passing to the limit in~(\ref{eq3A}) as
$t\rightarrow\infty,$ we obtain that $f_{n_k}(t)-t\rightarrow
-\frac{c_k}{r_0}.$ Since $f_{n_k}(t)-t\rightarrow 0$ as
$t\rightarrow \infty,$ we obtain that $c_k=0.$ By the Weierstrass
theorem on the convergence of the coefficients of the Taylor series
(see, e.g., \cite[Theorem~1.1.I]{Gol}) we obtain that
$c_k=0\rightarrow c_0$ as $k\rightarrow\infty.$ Thus, $c_0=0$
in~(\ref{eq4B}), in other words, the mapping $f$ also has a
hydrodynamic normalization: $f(z)=z+o(1)$ as $z\rightarrow\infty.$

\medskip
Now we show that $f$ is a homeomorphism of the complex plane. We put
$\mu_k:=\mu_{f_{n_k}}.$ By~\cite[Theorem~3.1]{LSS} $f_{n_k}$ is a
ring $Q$-homeomorphism at each point $z_0\in {\Bbb C},$ where
$Q=K_{\mu_k}(z).$
By the definition of the class of mappings
$\frak{F}^{\mathcal{M}}_{\Phi}(K),$ there exists $M_0>0$ such that
\begin{equation}\label{eq11}
\mathcal{M}({\Bbb C})\leqslant M_0\,,
\end{equation}
where $\mathcal{M}$ is a function from~(\ref{e3.3.1}). By
Lemma~\ref{lem1} the following alternative holds: either $f$ is a
homeomorphism with values in ${\Bbb C},$ or $f$ is a constant with
values in $\overline{\Bbb C}.$ As shown above in step~{\bf I}, the
mapping $f$ is homeomorphism in some neighborhood of infinity.
Therefore, $f$ is a homeomorphism of the whole complex plane, which
takes only finite complex values.

\medskip
By~\cite[Lemma~1 and Theorem~1]{L} $f\in W_{\rm loc}^{1, 1}({\Bbb
C}),$ beside that, $f$ is a regular solutions of the
equation~(\ref{eq2}) for some function $\mu:{\Bbb C}\rightarrow
{\Bbb D},$ and the relation~(\ref{e3.3.1}) is fulfilled for the
corresponding function $K_{\mu}.$ By the Gehring-Lehto theorem, the
map $f$ is almost everywhere differentiable
(see~\cite[Theorem~III.3.1]{LV}). Therefore, $\mu(z)=0$ for almost
all $z\in {\Bbb C}\setminus K$ (see Theorem~16.1 in~\cite{RSS}).
Thus, $f\in \frak{F}^{\mathcal{M}}_{\Phi}(K).$~$\Box$

\section{Equicontinuity of families of mappings with the inverse Poletsky inequality}

Compared to Theorem~\ref{th1}, Theorem~\ref{th2A} is a much more
complicated result, and for its proof we need a number of auxiliary
statements. The most important of them relate to mappings whose
inverse satisfy the inequality~(\ref{eq3*!gl0}). However, for
greater generality, we will establish the corresponding assertions
even in the case when the mappings under consideration have no
inverse at all, although, at the same time, the corresponding
inequality, containing the distortion of the modulus of families of
paths under the mappings, holds. Note that results similar to those
presented below were obtained by us in some individual situations
(see~\cite{SevSkv$_2$} and~\cite{SSD}).

\medskip
Let $y_0\in {\Bbb R}^n,$ $0<r_1<r_2<\infty,$ and let $A=A(y_0,
r_1,r_2)=\left\{ y\,\in\,{\Bbb R}^n: r_1<|y-y_0|<r_2\right\}$ is
defined in~(\ref{eq1F}).
Let, as above, $M(\Gamma)$ denotes a conformal modulus of families
of paths $\Gamma$ in ${\Bbb R}^n$ (see, e.g., \cite[Ch.~6]{Va}). Let
$f:D\rightarrow{\Bbb R}^n,$ $n\geqslant 2,$ be some mapping, and let
$Q:{\Bbb R}^n\rightarrow [0, \infty]$ be a Lebesgue measurable
function satisfying the condition $Q(x)\equiv 0$ for $x\in{\Bbb
R}^n\setminus f(D).$ Let $A=A(y_0, r_1, r_2),$ and let
$\Gamma_f(y_0, r_1, r_2)$ denotes the family of all paths
$\gamma:[a, b]\rightarrow D$ such that $f(\gamma)\in \Gamma(S(y_0,
r_1), S(y_0, r_2), A(y_0, r_1, r_2)),$ i.e., $f(\gamma(a))\in S(y_0,
r_1),$ $f(\gamma(b))\in S(y_0, r_2),$ and $f(\gamma(t))\in A(y_0,
r_1, r_2)$ for any $a<t<b.$ We say that {\it $f$ satisfies the
inverse Poletsky inequality} at $y_0\in f (D)$ if the relation
\begin{equation}\label{eq2*A}
M(\Gamma_f(y_0, r_1, r_2))\leqslant \int\limits_A Q(y)\cdot \eta^n
(|y-y_0|)\, dm(y)
\end{equation}
holds for any Lebesgue measurable function $\eta:(r_1,
r_2)\rightarrow [0, \infty]$ such that
\begin{equation}\label{eqA2}
\int\limits_{r_1}^{r_2}\eta(r)\,dr\geqslant 1 \,.
\end{equation}
Note that relations~(\ref{eq2*A})--(\ref{eqA2}) are equivalent to
relations~(\ref{eq3*!gl0})--(\ref{eq*3gl0}) for the inverse mapping
$g=f^{\,-1},$ provided that it exists. Indeed, observe that
\begin{equation}\label{eq2D}
g(\Gamma(S(y_0, r_1), S(y_0, r_2), f(D)))=\Gamma_f(y_0, r_1, r_2)\,.
\end{equation}
In fact, if $\gamma\in g(\Gamma(S(y_0, r_1), S(y_0, r_2), f(D))),$
then $\gamma:[a, b]\rightarrow {\Bbb R}^n,$ where
$\gamma=g\circ\alpha,$ $\alpha:[a, b]\rightarrow {\Bbb R}^n$ and
$\alpha(a)\in S(y_0, r_1),$ $\alpha(b)\in S(y_0, r_2),$
$\alpha(t)\in f(D)$ for $a\leqslant t\leqslant b.$ Now,
$\gamma(t)\in D$ for $a\leqslant t\leqslant b$ and
$f(\gamma)=\alpha\in \Gamma(S(y_0, r_1), S(y_0, r_2), f(D)),$ i.e.,
$\gamma\in \Gamma_f(y_0, r_1, r_2).$ Thus, $g(\Gamma(S(y_0, r_1),
S(y_0, r_2), f(D)))\subset\Gamma_f(y_0, r_1, r_2).$ The inverse
inclusion is proved similarly.

\medskip
A mapping $f$ between domains $D$ and $D^{\,\prime}$ is called {\it
closed} if $f(E)$ is closed in $D^{\,\prime}$ for any closed set
$E\subset D$ (see, e.g., \cite[Section~3]{Vu}). The boundary of the
domain $D$ is called {\it weakly flat at the point $x_0,$} if for
every number $P>0$ and for any neighborhood $U$ of this point there
is a neighborhood $V$ of $x_0$ such that $M(\Gamma(E, F, D))>P$ for
any continua $E$ and $F,$ satisfying conditions $F\cap
\partial U\ne\varnothing\ne F\cap
\partial V.$ The boundary of a domain
$D$ is called {\it weakly flat} if it is such at each point of its
boundary.

\medskip
Given $\delta>0,$ $M_0>0,$ domains $D, D^{\,\prime}\subset {\Bbb
R}^n,$ $n\geqslant 2,$ and a continuum $A\subset D^{\,\prime}$
denote by ${\frak S}_{\delta, A, M_0}(D, D^{\,\prime})$ the family
of all open discrete and closed mappings $f$ of $D$ onto
$D^{\,\prime},$ for which the condition~(\ref{eq2*A}) is fulfilled
for each $y_0\in D^{\,\prime}$ with some function $Q=Q_f\in
L^1(D^{\,\prime})$ such that
$\int\limits_{D^{\,\prime}}Q(y)\,dm(y)\leqslant M_0$ and, in
addition, $h(f^{\,-1}(A),
\partial D)\geqslant~\delta.$ An analogue of the
following statement is proved earlier for the case of a fixed
function $Q$ (see~\cite[Theorem~1.2]{SSD}). Note that the proof
given in~\cite{SSD} practically does not differ from the one given
below, however, for the sake of completeness of presentation, we
give it in full in this text. 

\medskip
\begin{theorem}\label{th2}
{\sl\, Let $D$ and $D^{\,\prime}$ be domains in ${\Bbb R}^n,$
$n\geqslant 2 .$ Assume that $D$ has a weakly flat boundary, and
$D^{\,\prime}$ is locally connected at the boundary. Then each
$f\in{\frak S}_{\delta, A, M_0}(D, D^{\,\prime})$ has a continuous
boundary extension $\overline{f}:\overline{D}\rightarrow
\overline{D^{\,\prime}}$ such that
$\overline{f}(\overline{D})=\overline{D^{\,\prime}},$ in addition,
the family ${\frak S}_{\delta, A, M_0}(\overline{D},
\overline{D^{\,\prime}})$ of all extended mappings
$\overline{f}:\overline{D}\rightarrow \overline{D^{\,\prime}}$ is
equicontinuous in $\overline{D}.$ }
\end{theorem}

\medskip
\begin{remark}\label{rem1}
In Theorem~\ref{th2}, the equicontinuity must be understood with
respect to the chordal metric, that is, for every $\varepsilon>0$
there is $\delta=\delta(\varepsilon, x_0)>0$ such that
$h(\overline{f}(x), \overline{f}(x_0))<\varepsilon$ for all $x\in
D,$ $h(x, x_0)<\delta,$ and all $\overline{f}\in{\frak S}_{\delta,
A, M_0}(\overline{D}, \overline{D^{\,\prime}}).$
\end{remark}

\medskip {\it Proof of Theorem~\ref{th2}.} Let $f\in {\frak
S}_{\delta, A, M_0}(D, D^{\,\prime}).$ By~\cite[Theorem~3.1]{SSD} a
mapping $f$ has a continuous extension
$\overline{f}:\overline{D}\rightarrow \overline{D^{\,\prime}},$
besides that, $\overline{f}(\overline{D})=\overline{D^{\,\prime}}.$
The equicontinuity of ${\frak S}_{\delta, A, M_0}(\overline{D},
\overline{D^{\,\prime}})$ in $D$ follows by~\cite[Theorem~1.1]{SSD}.
It remains to establish the equicontinuity of the family ${\frak
S}_{\delta, A, M_0}(\overline{D}, \overline{D^{\,\prime}})$ on
$\partial D.$

We carry out the proof by contradiction. Suppose that the above
conclusion does not hold. Then there is a point $x_0\in
\partial D,$ a positive number
$\varepsilon_0>0,$ a sequence $x_m\in \overline{D},$ converging to a
point $x_0$ and a map $\overline{f}_m\in {\frak S}_{\delta, A,
M_0}(\overline{D}, \overline{D})$ such that
\begin{equation}\label{eq12A}
h(\overline{f}_m(x_m),\overline{f}_m(x_0))\geqslant\varepsilon_0,\quad
m=1,2,\ldots .
\end{equation}
Put $f_m:=\overline{f}_m|_{D}.$ Since the map $f_m$ has a continuous
extension to $\partial D^{\,\prime},$ we may assume that $x_m\in D.$
Therefore, $\overline{f}_m(x_m)=f_m(x_m).$ In addition, there is a
sequence $x^{\,\prime}_m\in D$ such that $x^{\,\prime}_m\rightarrow
x_0$ as $m\rightarrow\infty$ and $h(f_m(x^{\,\prime}_m),
\overline{f}_m(x_0))\rightarrow 0$ as $m\rightarrow \infty.$ Since
$\overline{{\Bbb R}^n}$ is a compact metric space, we may consider
that $f_m(x_m)$ and $\overline{f}_m(x_0)$ converge as
$m\rightarrow\infty.$ Let $f_m(x_m)\rightarrow \overline{x_1}$ and
$\overline{f}_m(x_0)\rightarrow \overline{x_2}$ as
$m\rightarrow\infty.$ By the continuity of the metric
in~(\ref{eq12A}), $\overline{x_1}\ne \overline{x_2}.$ Without loss
of generality, we may consider that $\overline{x_1}\ne\infty.$ Since
$f_m$ are closed, they are boundary preserving
(see~\cite[Theorem~3.3]{Vu}). Thus, $\overline{x_2}\in\partial D.$
Let $\widetilde{x_1}$ and $\widetilde{x_2}$ be diferent points of
$A,$ none of which is the same as $\overline{x_1}.$
By~\cite[Lemma~2.1]{SevSkv$_1$}, cf.~\cite[Lemma~2.1]{SevSkv$_2$},
we can join pairs of points $\widetilde{x_1},$ $\overline{x_1}$ and
$\widetilde{x_2},$ $\overline{x_2}$ using paths $\gamma_1:[0,
1]\rightarrow \overline{D^{\,\prime}}$ and $\gamma_2:[0,
1]\rightarrow \overline{D^{\,\prime}}$ so that $|\gamma_1|\cap
|\gamma_2|=\varnothing,$ $\gamma_1(t), \gamma_2(t)\in D^{\,\prime}$
for $t\in (0, 1),$ $\gamma_1(0)=\widetilde{x_1},$
$\gamma_1(1)=\overline{x_1},$ $\gamma_2(0)=\widetilde{x_2}$ and
$\gamma_2(1)=\overline{x_2}.$ Since $D^{\,\prime}$ is locally
connected on $\partial D^{\,\prime},$  there are neighborhoods $U_1$
and $U_2$ of points $\overline{x_1}$ and $\overline{x_2},$
respectively, whose closures do not intersect, and, moreover, sets
$W_i:=D^{\,\prime}\cap U_i$ are path-connected. Without loss of
generality, we may assume that $\overline{U_1}\subset
B(\overline{x_1}, \delta_0)$ and
\begin{equation}\label{eq12C}
\overline{B(\overline{x_1},
\delta_0)}\cap|\gamma_2|=\varnothing=\overline{U_2}\cap|\gamma_1|\,,
\quad \overline{B(\overline{x_1}, \delta_0)}\cap
\overline{U_2}=\varnothing\,.
\end{equation}
We may also consider that $f_m(x_m)\in W_1$ and
$f_m(x^{\,\prime}_m)\in W_2$ for any $m\in {\Bbb N}.$ Let $a_1$ and
$a_2$ are arbitrary points in $|\gamma_1|\cap W_1$ and
$|\gamma_2|\cap W_2,$ correspondingly. Let $0<t_1, t_2<1$ be such
that $\gamma_1(t_1)=a_1$ and $\gamma_2(t_2)=a_2.$ Join the points
$a_1$ and $f_m(x_m)$ by means of the path $\alpha_m:[t_1,
1]\rightarrow W_1$ such that $\alpha_m(t_1)=a_1$ and
$\alpha_m(1)=f_m(x_m).$ Similarly, let us join $a_2$ and
$f_m(x^{\,\prime}_m)$ by means of the path $\beta_m:[t_2,
1]\rightarrow W_2,$ $\beta_m(t_2)=a_2$ and
$\beta_m(1)=f_m(x^{\,\prime}_m)$ (see Figure~\ref{fig4}).
\begin{figure}[h]
\centerline{\includegraphics[scale=0.55]{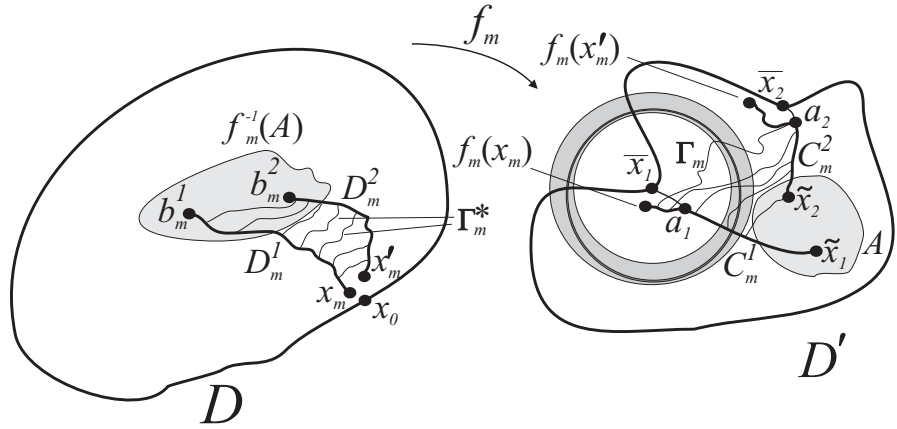}} \caption{To
the proof of Theorem~\ref{th2}}\label{fig4}
\end{figure}
Put
$$C^1_m(t)=\quad\left\{
\begin{array}{rr}
\gamma_1(t), & t\in [0, t_1],\\
\alpha_m(t), & t\in [t_1, 1]\end{array} \right.\,,\qquad
C^2_m(t)=\quad\left\{
\begin{array}{rr}
\gamma_2(t), & t\in [0, t_2],\\
\beta_m(t), & t\in [t_2, 1]\end{array} \right.\,.$$
Recall that a path $\alpha:[a, b]\rightarrow D$ is called a {\it
(whole) $f$-lifting of $\beta:[a, b]\rightarrow {\Bbb R}^n$ starting
at $x_0\in D,$} if $f(\alpha(t))=\beta(t)$ for any $t\in[a, b]$ and,
in addition, $\alpha(a)=x_0.$ Let $D^1_m$ and $D^2_m$ be whole
$f_m$-liftings of paths $C^1_m$ and $C^2_m$ starting at points $x_m$
and $x^{\,\prime}_m,$ respectively (such liftings exist
by~\cite[Lemma~3.7]{Vu}). In particular, by the condition
$h(f_m^{\,-1}(A),
\partial D)\geqslant~\delta>0$ from the definition of a class~${\frak
S}_{\delta, A, M_0}(D, D^{\,\prime}),$ the end points of paths
$D^1_m$ and $D^2_m,$ which we denote by $b_m^1$ and $b_m^ 2,$
distant from the boundary $D$ by at least on $\delta.$ As usual, we
denote by $|C^1_m|$ and $|C^2_m|$ the loci of the paths $C^1_m$ and
$C^2_m,$ respectively. Set
$$l_0=\min\{{\rm dist}\,(|\gamma_1|,
|\gamma_2|), {\rm dist}\,(|\gamma_1|, U_2\setminus\{\infty\})\}\,.$$
Consider the coverage $A_0:=\bigcup\limits_{x\in |\gamma_1|}B(x,
l_0/4)$ by balls of $|\gamma_1|.$ Since $|\gamma_1|$ is a compact
set, we may choose a finite number of indexes $1\leqslant
N_0<\infty$ and the corresponding points $z_1,\ldots, z_{N_0}\in
|\gamma_1|$ such that $|\gamma_1|\subset
B_0:=\bigcup\limits_{i=1}^{N_0}B(z_i, l_0/4).$
In this case,
$$|C^1_m|\subset U_1\cup |\gamma_1|\subset
B(\overline{x_1}, \delta_0)\cup \bigcup\limits_{i=1}^{N_0}B(z_i,
l_0/4)\,.$$
Let $\Gamma_m$ be a family of paths joining the sets $|C^1_m|$ and
$|C^2_m|$ in $D^{\,\prime}.$ Now we obtain that
\begin{equation}\label{eq10C}
\Gamma_m=\bigcup\limits_{i=0}^{N_0}\Gamma_{mi}\,,
\end{equation}
where $\Gamma_{mi}$ is a family of all paths $\gamma:[0,
1]\rightarrow D^{\,\prime}$ such that $\gamma(0)\in B(z_i,
l_0/4)\cap |C^1_m|$ and $\gamma(1)\in |C_2^m|$ for $1\leqslant
i\leqslant N_0.$ Similarly, let $\Gamma_{m0}$ be a family of all
paths $\gamma:[0, 1]\rightarrow D^{\,\prime}$ such that
$\gamma(0)\in B(\overline{x_1}, \delta_0)\cap |C^1_m|$ and
$\gamma(1)\in |C_2^m|.$ By~(\ref{eq12C}) there is some
$\sigma_0>\delta_0>0$ such that
$$
\overline{B(\overline{x_1},
\sigma_0)}\cap|\gamma_2|=\varnothing=\overline{U_2}\cap|\gamma_1|\,,
\quad \overline{B(\overline{x_1}, \sigma_0)}\cap
\overline{U_2}=\varnothing\,.$$
By~\cite[Theorem~1.I.5.46]{Ku}, we may show that
$$\Gamma_{m0}>\Gamma(S(\overline{x_1}, \delta_0), S(\overline{x_1}, \sigma_0), A(\overline{x_1}, \delta_0,
\sigma_0))\,,$$
\begin{equation}\label{eq11C}
\Gamma_{mi}>\Gamma(S(z_i, l_0/4), S(z_i, l_0/2), A(z_i, l_0/4,
l_0/2))\,.
\end{equation}
We may chose $x_*\in D^{\,\prime},$ $\delta_*>0$ and $\sigma_*>0$
such that $A(x_*, \delta_*, \sigma_*)\subset A(\overline{x_1},
\delta_0, \sigma_0).$ Thus,
$$\Gamma(S(\overline{x_1}, \delta_0), S(\overline{x_1}, \sigma_0),
A(\overline{x_1}, \delta_0, \sigma_0))>$$
\begin{equation}\label{eq1E*}
>\Gamma(S(x_*, \delta_*),
S(x_*, \sigma_*), A(x_*, \delta_*, \sigma_*))\,.
\end{equation}
Set
$$\eta(t)= \left\{
\begin{array}{rr}
4/l_0, & t\in [l_0/4, l_0/2],\\
0,  &  t\not\in [l_0/4, l_0/2]\,,
\end{array}
\right. $$$$\eta_0(t)= \left\{
\begin{array}{rr}
1/(\sigma_*-\delta_*), & t\in [\delta_*, \sigma_*],\\
0,  &  t\not\in [\delta_*, \sigma_*]\,.
\end{array}
\right.$$
Let $\Gamma^{\,*}_m:=\Gamma(|D_m^1|, |D_m^2|, D).$ Observe that
$f_m(\Gamma^{\,*}_m)\subset\Gamma_m.$ Now, by~(\ref{eq10C}),
(\ref{eq11C}) and~(\ref{eq1E*}) we obtain that
\begin{equation}\label{eq6A}
\Gamma^{\,*}_m>\left(\bigcup\limits_{i=1}^{N_0}\Gamma_{f_m}(z_i,
l_0/4, l_0/2)\right)\cup \Gamma_{f_m}(x_*, \delta_*, \sigma_*)\,.
\end{equation}
Since~$f_m$ satisfy the relation~(\ref{eq2*A}) for $Q=Q_m$ in
$D^{\,\prime},$ we obtain by~(\ref{eq6A}) that
\begin{equation}\label{eq14A}
M(\Gamma^{\,*}_m)\leqslant
(4^nN_0/l_0^n+(1/(\sigma_*-\delta_*))^n)\Vert Q_m\Vert_1\leqslant
c<\infty\,.
\end{equation}
We show that the relation~(\ref{eq14A}) contradicts the condition of
the weakly flatness of the mapped domain. In fact, by construction
$$h(|D^1_m|)\geqslant h(x_m, b_m^1) \geqslant
(1/2)\cdot h(f^{\,-1}_m(A), \partial D)>\delta/2\,,$$
\begin{equation}\label{eq14}
h(|D^2_m|)\geqslant h(x^{\,\prime}_m, b_m^2) \geqslant (1/2)\cdot
h(f^{\,-1}_m(A), \partial D)>\delta/2
\end{equation}
for all $m\geqslant \widetilde{M_0}$ and some $\widetilde{M_0}\in
{\Bbb N}.$
We put~$U:=B_h(x_0, r_0)=\{y\in \overline{{\Bbb R}^n}: h(y,
x_0)<r_0\},$ where $0<r_0<\delta/4$ and the number $\delta $ refers
to the ratio~(\ref{eq14}). Note that $|D^1_m|\cap U\ne\varnothing\ne
|D^1_m|\cap (D\setminus U)$ for every $m\in{\Bbb N},$ because
$h(|D^1_m|)\geqslant \delta/2$ and $x_m\in |D^1_m|,$ $x_m\rightarrow
x_0$ as $m\rightarrow\infty.$ Similarly, $|D^2_m|\cap
U\ne\varnothing\ne |D^2_m|\cap (D\setminus U).$ Since $|D^1_m|$ and
$|D^2_m|$ are continua,
\begin{equation}\label{eq8A}
|D^1_m|\cap \partial U\ne\varnothing, \quad |D^2_m|\cap
\partial U\ne\varnothing\,,
\end{equation}
see, e.g.,~\cite[Theorem~1.I.5.46]{Ku}. Let $c$ be the number
from~(\ref{eq14A}). Since $\partial D$ is weakly flat, for $P:=c>0$
there is a neighborhood $V\subset U$ of $x_0$ such that
\begin{equation}\label{eq9A}
M(\Gamma(E, F, D))>c
\end{equation}
for any continua $E, F\subset D$ such that $E\cap
\partial U\ne\varnothing\ne E\cap \partial V$ and $F\cap \partial
U\ne\varnothing\ne F\cap \partial V.$ Let us show that
\begin{equation}\label{eq10A}
|D^1_m|\cap \partial V\ne\varnothing, \quad |D^2_m|\cap
\partial V\ne\varnothing\end{equation}
for sufficiently large $m\in {\Bbb N}.$ Indeed, $x_m\in |D^1_m|$ and
$x^{\,\prime}_m\in |D^2_m|,$ where $x_m, x^{\,\prime}_m\rightarrow
x_0\in V$ as $m\rightarrow\infty.$ In this case, $|D^1_m|\cap
V\ne\varnothing\ne |D^2_m|\cap V$ for sufficiently large $m\in {\Bbb
N}.$ Observe that $h(V)\leqslant h(U)\leqslant 2r_0<\delta/2.$
By~(\ref{eq14}) $h(|D^1_m|)>\delta/2.$ Thus, $|D^1_m|\cap
(D\setminus V)\ne\varnothing$ and, therefore, $|D^1_m|\cap\partial
V\ne\varnothing$ (see, e.g.,~\cite[Theorem~1.I.5.46]{Ku}).
Similarly, $h(V)\leqslant h(U)\leqslant 2r_0<\delta/2.$
By~(\ref{eq14}) we obtain that $h(|D^2_m|)>\delta/2.$ Thus,
$|D^2_m|\cap (D\setminus V)\ne\varnothing.$
By~\cite[Theorem~1.I.5.46]{Ku} we obtain that $|D^2_m|\cap\partial
V\ne\varnothing.$ Thus, the relation~(\ref{eq10A}) is proved.
Combining the relations~(\ref{eq8A}), (\ref{eq9A}) and
(\ref{eq10A}), we obtain that $M(\Gamma^{\,*}_m)=M(\Gamma(|D^1_m|,
|D^2_m|, D))>c.$ The last relation contradicts the
inequality~(\ref{eq14A}), which proves the theorem.~$\Box$

\medskip
The following lemma was also proved earlier in somewhat other
situations, in particular, in the case of a fixed function $Q$ (see,
e.g., \cite[Lemma~4.1]{SevSkv$_2$}, \cite[Lemma~4.1]{SevSkv$_1$}
and~\cite[Lemma~6.1]{SSD}).

\medskip
\begin{lemma}\label{lem3}
{\sl\, Let $n\geqslant 2,$ and let $D$ and $D^{\,\prime}$ be domains
in ${\Bbb R}^n$ such that $D^{\,\prime}$ is locally connected on
$\partial D^{\,\prime},$ $D$ has a weakly flat boundary, and,
moreover, no connected component of $\partial D$ does not degenerate
into a point. Let $A$ be a nondegenerate continuum in
$D^{\,\prime},$ and let $\delta, M> 0.$ Assume that $f_m,$
$m=1,2,\ldots,$ be a sequence of discrete, open and closed mappings
of $D$ onto $D^{\,\prime},$ satisfying the relation~(\ref{eq2*A}) in
$D$ for some function $Q=Q_m,$ such that $f_m(A_m)=A$ for some
continuum $A_m\subset D$ with $h(A_m)\geqslant \delta>0,$
$m=1,2,\ldots .$ If $\int\limits_{D^{\,\prime}}Q(y)\,dm(y)\leqslant
M_0$ for $m=1,2,\ldots ,$ then there is $\delta_1>0$ such that
$$h(A_m,
\partial D)>\delta_1>0\quad \forall\,\, m\in {\Bbb
N}\,.$$
}
\end{lemma}
\begin{proof}
Due to the compactness of the space~$\overline{{\Bbb R}^n}$ the
boundary of the domain $D$ is not empty and compact, so that the
distance $h(A_m,
\partial D)$ is well-defined.

\medskip
We carry out the proof by contradiction. Suppose that the conclusion
of the lemma is not true. Then for each $k\in {\Bbb N}$ there is
some number $m=m_k\in {\Bbb N}$ such that $h(A_{m_k},
\partial D)<1/k.$ We may assume that $m_k$ increases on $k=1,2,\ldots .$
Since $A_{m_k}$ is a compact set, there are $x_k\in A_{m_k}$ and
$y_k\in
\partial D,$ $k=1,2,\ldots ,$ such that $h(A_{m_k},
\partial D)=h(x_k, y_k)<1/k$ (see Figure~\ref{fig3}).
\begin{figure}[h]
\centerline{\includegraphics[scale=0.6]{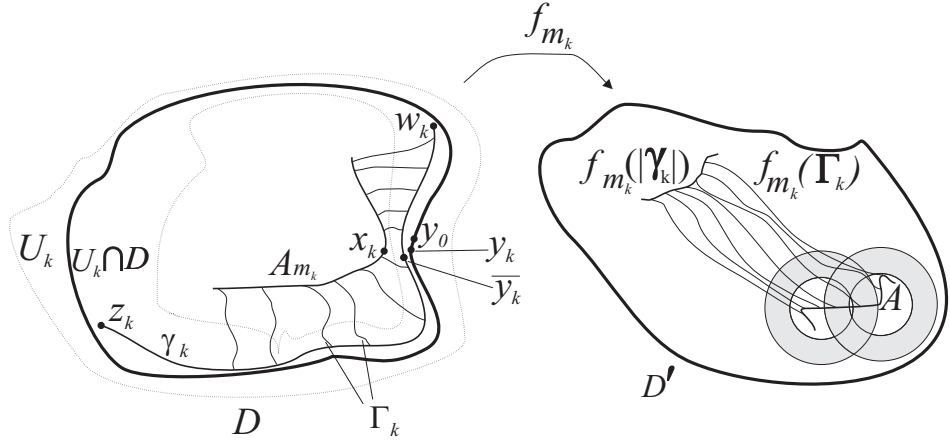}} \caption{To
the proof of Lemma~\ref{lem3}}\label{fig3}
\end{figure}
Since $\partial D$ is a compact set, we may consider that
$y_k\rightarrow y_0\in
\partial D$ as $k\rightarrow \infty.$ Then
%
$x_k\rightarrow y_0\in \partial D$ при $k\rightarrow \infty.$
%
Let $K_0$ be a connected component of $\partial D,$ containing
$y_0.$ Obviously, $K_0$ is a continuum in $\overline{{\Bbb R}^n}.$
Since $\partial D$ is weakly flat, by~\cite[Theorem~3.1]{SSD}
$f_{m_k}$ has a continuous extension
$\overline{f}_{m_k}:\overline{D}\rightarrow
\overline{D^{\,\prime}}.$ Moreover, $\overline{f}_{m_k}$ is
uniformly continuous in $\overline{D}$ for any fixed $k,$ because
$\overline{f}_{m_k}$ is continuous on a compact set~$\overline{D}.$
Now, for every $\varepsilon>0$ there is
$\delta_k=\delta_k(\varepsilon)<1/k$ such that
\begin{equation}\label{eq3B}
h(\overline{f}_{m_k}(x),\overline{f}_{m_k}(x_0))<\varepsilon
\end{equation}
$$
\forall\,\, x,x_0\in \overline{D},\quad h(x, x_0)<\delta_k\,, \quad
\delta_k<1/k\,.$$
Pick $\varepsilon>0$ such that
\begin{equation}\label{eq5D}
\varepsilon<(1/2)\cdot h(\partial D^{\,\prime}, A)\,.
\end{equation}
Denote $B_h(x_0, r)=\{x\in \overline{{\Bbb R}^n}: h(x, x_0)<r\}.$
Given $k\in {\Bbb N},$ put
$$B_k:=\bigcup\limits_{x_0\in K_0}B_h(x_0, \delta_k)\,,\quad k\in {\Bbb
N}\,.$$
Since $B_k$ is a neighborhood of the continuum $K_0,$
by~\cite[Lemma~2.2]{HK} there is a neighborhood $U_k$ of $K_0$ such
that $U_k\subset B_k$ and $U_k\cap D$ is connected. We may consider
that $U_k$ is open, so that $U_k\cap D$ is path connected
(see~\cite[Proposition~13.1]{MRSY}). Let $h(K_0)=m_0.$ Now, we can
find $z_0, w_0\in K_0$ such that $h(K_0)=h(z_0, w_0)=m_0.$ Thus,
there are sequences $\overline{y_k}\in U_k\cap D,$ $z_k\in U_k\cap
D$ and $w_k\in U_k\cap D$ such that $z_k\rightarrow z_0,$
$\overline{y_k}\rightarrow y_0$ and $w_k\rightarrow w_0$ as
$k\rightarrow\infty.$ We may consider that
\begin{equation}\label{eq2G}
h(z_k, w_k)>m_0/2\quad \forall\,\, k\in {\Bbb N}\,.
\end{equation}
Since $U_k\cap D$ is path connected, we can join the points $z_k,$
$\overline{y_k}$ and $w_k$ using some path $\gamma_k\in U_k\cap D.$
As usual, we denote by $|\gamma_k|$ a locus of the path $\gamma_k$
in $D.$ In this case, $f_{m_k}(|\gamma_k|)$ is a compact set
in~$D^{\,\prime}.$ If $x\in|\gamma_k|,$ there exists~$x_0\in K_0$
such that $x\in B(x_0, \delta_k).$ Put $\omega\in A\subset D.$ Since
$x\in|\gamma_k|$ and, in addition, $x$ is an inner point of $D,$ we
may use the notation $f_{m_k}(x)$ instead of
$\overline{f}_{m_k}(x).$ By~(\ref{eq3B}) and~(\ref{eq5D}), and by
the triangle inequality, we obtain that
$$h(f_{m_k}(x),\omega)\geqslant
h(\omega,
\overline{f}_{m_k}(x_0))-h(\overline{f}_{m_k}(x_0),f_{m_k}(x))\geqslant$$
\begin{equation}\label{eq4C}
\geqslant h(\partial D^{\,\prime}, A)-(1/2)\cdot h(\partial
D^{\,\prime}, A)=(1/2)\cdot h(\partial D^{\,\prime}, A)>\varepsilon
\end{equation}
for sufficiently large $k\in {\Bbb N}.$ Taking $\inf$ in the
relation~(\ref{eq4C}) over all $x\in |\gamma_k|$ and $\omega\in A,$
we obtain that
$h(f_{m_k}(|\gamma_k|), A)>\varepsilon,$ $k=1,2,\ldots .$ Since
$h(x, y)\leqslant |x-y|$ for any $x, y\in {\Bbb R}^n,$ we obtain
that
\begin{equation}\label{eq6B}
{\rm dist}\,(f_{m_k}(|\gamma_k|), A)>\varepsilon, \quad\forall\,\,
k=1,2,\ldots \,.
\end{equation}
We cover the continuum $A$ with balls $B(x, \varepsilon/4),$ $x\in
A.$ Since $A$ is a compact set, we may consider that $A\subset
\bigcup\limits_{i=1}^{N_0}B(x_i, \varepsilon/4),$ $x_i\in A,$
$i=1,2,\ldots, N_0,$ $1\leqslant N_0<\infty.$ By the definition,
$N_0$ depends only on $A,$ in particular, $N_0$ does not depend
on~$k.$ Set
\begin{equation}\label{eq5C}
\Gamma_k:=\Gamma(A_{m_k}, |\gamma_k|, D)\,.
\end{equation}
Let $\Gamma_{ki}:=\Gamma_{f_{m_k}}(x_i, \varepsilon/4,
\varepsilon/2),$ in other words, $\Gamma_{ki}$ consists from all
paths $\gamma:[0, 1]\rightarrow D$ such that $f_{m_k}(\gamma(0))\in
S(x_i, \varepsilon/4),$ $f_{m_k}(\gamma(1))\in S(x_i,
\varepsilon/2)$ and $\gamma(t)\in A(x_i, \varepsilon/4,
\varepsilon/2)$ for $0<t<1.$ Let us show that
\begin{equation}\label{eq6C}
\Gamma_k>\bigcup\limits_{i=1}^{N_0}\Gamma_{ki}\,.
\end{equation}
Indeed, let $\widetilde{\gamma}\in \Gamma_k,$ in other words,
$\widetilde{\gamma}:[0, 1]\rightarrow D,$ $\widetilde{\gamma}(0)\in
A_{m_k},$ $\widetilde{\gamma}(1)\in |\gamma_k|$ and
$\widetilde{\gamma}(t)\in D$ for $0\leqslant t\leqslant 1.$ Then
$\gamma^{\,*}:=f_{m_k}(\widetilde{\gamma})\in \Gamma(A,
f_{m_k}(|\gamma_k|), D^{\,\prime}).$ Since the balls $B(x_i,
\varepsilon/4),$ $1\leqslant i\leqslant N_0,$ form a covering of the
compactum $A,$ there is $i\in {\Bbb N}$ such that
$\gamma^{\,*}(0)\in B(x_i, \varepsilon/4)$ and $\gamma^{\,*}(1)\in
f_{m_k}(|\gamma_k|).$ By~(\ref{eq6B}), $|\gamma^{\,*}|\cap B(x_i,
\varepsilon/4)\ne\varnothing\ne |\gamma^{\,*}|\cap
(D^{\,\prime}\setminus B(x_i, \varepsilon/4)).$ Thus,
by~\cite[Theorem~1.I.5.46]{Ku} there is $0<t_1<1$ such that
$\gamma^{\,*}(t_1)\in S(x_i, \varepsilon/4).$ We may consider that
$\gamma^{\,*}(t)\not\in B(x_i, \varepsilon/4)$ for $t>t_1.$ Set
$\gamma_1:=\gamma^{\,*}|_{[t_1, 1]}.$ By~(\ref{eq6B}) it follows
that $|\gamma_1|\cap B(x_i, \varepsilon/2)\ne\varnothing\ne
|\gamma_1|\cap (D\setminus B(x_i, \varepsilon/2)).$ Thus,
by~\cite[Theorem~1.I.5.46]{Ku} there is $t_1<t_2<1$ such that
$\gamma^{\,*}(t_2)\in S(x_i, \varepsilon/2).$ We may consider that
$\gamma^{\,*}(t)\in B(x_i, \varepsilon/2)$ for any $t<t_2.$ Putting
$\gamma_2:=\gamma^{\,*}|_{[t_1, t_2]},$ we observe that $\gamma_2$
is a subpath of $\gamma^{\,*},$ which belongs to $\Gamma(S(x_i,
\varepsilon/4), S(x_i, \varepsilon/2), A(x_i, \varepsilon/4,
\varepsilon/2)).$

Finally, $\widetilde{\gamma}$ has a subpath
$\widetilde{\gamma_2}:=\widetilde{\gamma}|_{[t_1, t_2]}$ such that
$f_{m_k}\circ\widetilde{\gamma_2}=\gamma_2,$ wherein
$$\gamma_2\in \Gamma(S(x_i, \varepsilon/4), S(x_i, \varepsilon/2),
A(x_i, \varepsilon/4, \varepsilon/2))\,.$$ Thus, the
relation~(\ref{eq6C}) is fulfilled. Set
$$\eta(t)= \left\{
\begin{array}{rr}
4/\varepsilon, & t\in [\varepsilon/4, \varepsilon/2],\\
0,  &  t\not\in [\varepsilon/4, \varepsilon/2]\,.
\end{array}
\right. $$
Observe that $\eta$ satisfies~(\ref{eqA2}) for $r_1=\varepsilon/4$
and $r_2=\varepsilon/2.$ Since a mapping $f_{m_k}$ satisfies the
relation~(\ref{eq2*A}) for $Q=Q_{m_k}$, $k=1,2,\ldots ,$ putting
here $y_0=x_i,$ we obtain that
\begin{equation}\label{eq8C}
M(\Gamma_{f_{m_k}}(x_i, \varepsilon/4, \varepsilon/2))\leqslant
(4/\varepsilon)^n\cdot\Vert Q_{m_k}\Vert_1\leqslant
(4/\varepsilon)^nM_0 <\infty\,,
\end{equation}
where $\Vert Q_{m_k}\Vert_1$ is $L_1$-norm of the function $Q_{m_k}$
in $D^{\,\prime}.$ By~(\ref{eq6C}) and (\ref{eq8C}), taking into
account the subadditivity of the modulus of families of paths, we
obtain that
\begin{equation}\label{eq4B*}
M(\Gamma_k)\leqslant
\frac{4^nN_0}{\varepsilon^n}\int\limits_{D^{\,\prime}}Q_{m_k}(y)\,dm(y)\leqslant
c<\infty\,,
\end{equation}
where $c$ is some positive constant. Arguing similarly to the proof
of relations~(\ref{eq14}), and using the condition~(\ref{eq2G}), we
obtain that $M(\Gamma_k)\rightarrow\infty$ as $k\rightarrow\infty,$
that contradicts to~(\ref{eq4B*}). The contradiction obtained above
prove the lemma.~$\Box$
\end{proof}

\medskip
Given domains $D, D^{\,\prime}\subset {\Bbb R}^n,$ points $a\in D,$
$b\in D^{\,\prime}$ and a number $M_0>0 $ denote by ${\frak S}_{a,
b, M_0}(D, D^{\,\prime})$ the family of all open, discrete, and
closed mappings $f$ of $D$ onto $D^{\,\prime}$ satisfying the
condition~(\ref{eq2*A}) with some $Q=Q_f,$ $\Vert
Q\Vert_{L^1(D^{\,\prime})}\leqslant M_0,$ for every $y_0\in f(D),$
such that $f(a)=b.$ The following statement is proved
in~\cite[Theorem~7.1]{SSD} in the case of a fixed function~$Q.$

\medskip
\begin{theorem}\label{th4}
{\sl Suppose that a domain $D$ has a weakly flat boundary, none of
the connected components of which is degenerate. If $D^{\,\prime}$
is locally connected at its boundary, then each map $f\in {\frak
S}_{a, b, M_0}(D, D^{\,\prime})$ has a continuous boundary extension
$\overline{f}:\overline{D}\rightarrow \overline{D^{\,\prime}}$ such
that $\overline{f}(\overline{D})=\overline{D^{\,\prime}}$ and, in
addition, the family ${\frak S}_{a, b, M_0}(\overline{D},
\overline{D^{\,\prime}})$ of all extended mappings
$\overline{f}:\overline{D}\rightarrow \overline{D^{\,\prime}}$ is
equicontinuous in $\overline{D}.$ }
\end{theorem}

\medskip
\begin{proof} The equicontinuity of the family ${\frak S}_{a, b, M_0}(D,
D^{\,\prime}),$ the possibility of continuous boundary extension of
$f\in{\frak S}_{a, b, M_0}(D, D^{\,\prime})$ and the equality
$\overline{f}(\overline{D})=\overline{D^{\,\prime}}$ follow from
~\cite[Theorems~1.1 and 3.1]{SSD}. It remains to establish the
equicontinuity of the family of extended mappings
$\overline{f}:\overline{D}\rightarrow \overline{D^{\,\prime}}$ at
the boundary points of the domain~$D.$

\medskip
We will prove this statement from the contrary. Assume that the
family ${\frak S}_{a, b, M_0}(\overline{D},
\overline{D^{\,\prime}})$ is not equicontinuous at some point
$x_0\in\partial D.$ Then there are points $x_m\in D,$ $m=1,2,\ldots
,$ and mappings $f_m\in {\frak S}_{a, b, M_0}(\overline{D},
\overline{D^{\,\prime}}),$ such that $x_m\rightarrow x_0$ as
$m\rightarrow\infty,$ and
\begin{equation}\label{eq15}
h(f_m(x_m), f_m(x_0))\geqslant\varepsilon_0\,,\quad m=1,2,\ldots\,,
\end{equation}
for some $\varepsilon_0>0.$ Choose $y_0\in D^{\,\prime},$ $y_0\ne
b,$ and join it with the point $b$ by some path $\alpha$ in
$D^{\,\prime}.$ Put $A:=|\alpha|.$ Let $A_m$ be a whole
$f_m$-lifting of $\alpha $ starting at $a$ (it exists
by~\cite[Lemma~3.7]{Vu}). Note that $h(A_m,
\partial D)>0$ by the closeness of the mapping~$f_m.$
The following cases are now possible: either $h(A_m)\rightarrow 0$
as $m\rightarrow\infty,$ or $h(A_{m_k})\geqslant\delta_0>0$ as
$k\rightarrow\infty$ for some increasing sequence of numbers $m_k,$
$k=1,2,\ldots, $ and some $\delta_0>0.$

In the first case, obviously, $h(A_m, \partial D)\geqslant \delta>0$
for some $\delta>0.$ Now, the family ${\{f_m\}}^{\infty}_{m=1}$ is
equicontinuous at $x_0$ by Theorem~\ref{th2}, which contradicts the
condition~(\ref{eq15}).

In the second case, if $h(A_{m_k})\geqslant\delta_0>0$ as
$k\rightarrow\infty,$ then by Lemma~\ref{lem3} we also have that
$h(A_{m_k}, \partial D)\geqslant \delta_1>0$ for some $\delta_1> 0.$
Again, by Theorem~\ref{th2}, the family
${\{f_{m_k}\}}^{\infty}_{k=1}$ is equicontinuous at $x_0,$ and this
contradicts the condition~(\ref{eq15}).

The contradiction obtained above indicates that the assumption
concerning the absence of equicontinuity of the family ${\frak
S}_{a, b, M_0}(D, D^{\,\prime})$ on $\overline{D}$ is incorrect.
Theorem proved.~$\Box$
\end{proof}

\section{Proof of Theorem~\ref{th2A}}

{\textbf I.} Let $f_m\in \frak{F}^{\mathcal{M}}_{\varphi, \Phi,
z_0}(D),$ $m=1,2,\ldots .$ By Stoilow's factorization theorem (see,
e.g., \cite[5(III).V]{St}) a mapping $f_m$ has a representation
\begin{equation}\label{eq2E}
f_m=\varphi_m\circ g_m\,,
\end{equation}
where $g_m$ is some homeomorphism, and $\varphi_m$ is some analytic
function. By Lemma~1 in~\cite{Sev$_2$}, the mapping $g_m$ belongs to
the Sobolev class $W_{\rm loc}^{1, 1}(D)$ and has a finite
distortion. Moreover, by~\cite[(1).C, Ch.~I]{A}
\begin{equation}\label{eq1B}
{f_m}_z={\varphi_m}_z(g_m(z)){g_m}_z,\qquad
{f_m}_{{\overline{z}}}={\varphi_m}_z(g_m(z)){g_m}_{\overline{z}}
\end{equation}
for almost all $z\in D.$ Therefore, by the relation~(\ref{eq1B}),
$J(z, g_m)\ne 0$ for almost all $z\in D,$ in addition,
$K_{\mu_{f_m}}(z)=K_{\mu_{g_m}}(z).$

\medskip
\textbf{II.} We prove that $\partial g_m (D)$ contains at least two
points. Suppose the contrary. Then either $g_m(D)={\Bbb C},$ or
$g_m(D)={\Bbb C}\setminus \{a\},$ where $a\in {\Bbb C}.$ Consider
first the case $g_m(D)={\Bbb C}.$ By Picard's theorem
$\varphi_m(g_m(D))$ is the whole plane, except perhaps one point
$\omega_0\in {\Bbb C}.$ On the other hand, for every $m=1,2,\ldots$
the function $u_m(z):={\rm Re}\,f_m(z)={\rm
Re}\,(\varphi_m(g_m(z)))$ is continuous on the compact set
$\overline{D}$ under the condition~(\ref{eq1A}) by the continuity
of~$\varphi.$ Therefore, there exists $C_m>0$ such that $|{\rm
Re}\,f_m(z)|\leqslant C_m$ for any $z\in D,$ but this contradicts
the fact that $\varphi_m(g_m(D))$ contains all points of the complex
plane except, perhaps, one. The situation $g_m(D)={\Bbb C}\setminus
\{a\},$ $a\in {\Bbb C},$ is also impossible, since the domain
$g_m(D)$ must be simply connected in ${\Bbb C}$ as a homeomorphic
image of the simply connected domain $D.$

\medskip
Therefore, the boundary of the domain $g_m(D)$ contains at least two
points. Then, according to Riemann's mapping theorem, we may
transform the domain $g_m(D)$ onto the unit disk ${\Bbb D}$ using
the conformal mapping $\psi_m.$ Let $z_0\in D $ be a point from the
condition of the theorem. By using an auxiliary conformal mapping
$$\widetilde{\psi_m}(z)=\frac{z-(\psi_m\circ
g_m)(z_0)}{1-z\overline{(\psi_m\circ g_m)(z_0)}}$$ of the unit disk
onto itself we may consider that $(\psi_m\circ g_m)(z_0)=0.$ Now,
by~(\ref{eq2E}) we obtain that
\begin{equation}\label{eq2F}
f_m=\varphi_m\circ g_m= \varphi_m\circ\psi^{\,-1}_m\circ\psi_m\circ
g_m=F_m\circ G_m\,,\quad m=1,2,\ldots\,,
\end{equation}
where $F_m:=\varphi_m\circ\psi^{\,-1}_m,$ $F_m:{\Bbb D}\rightarrow
{\Bbb C},$ and $G_m=\psi_m\circ g_m.$
Obviously, a function $F_m$ is analytic, and $G_m$ is a regular
Sobolev homeomorphism in $D.$ In particular, ${\rm Im}\,F_m(0)=0$
for any $m\in {\Bbb N}.$

\medskip
\textbf{III.} We prove that the $L^1$-norms of the functions
$K_{\mu_{G_m}}(z)$ are bounded from above by some universal positive
constant $C> 0$ over all $m=1,2,\ldots .$ Indeed, by the convexity
of the function $\Phi$ in~(\ref{eq1D}) and by~\cite[Proposition~5,
I.4.3]{Bou},  the slope $\left[\Phi(t)-\Phi(0)\right]/t$ is a
non-decreasing function. Hence there exist constants $t_0>0$ and
$C_1>0$ such that
\begin{equation}\label{eq6D}
\Phi(t)\geqslant C_1\cdot t\qquad \forall\,\, t\in [t_0, \infty)\,.
\end{equation}
Fix $m\in {\Bbb N}\,.$ By~(\ref{eq1D}) and~(\ref{eq6D}), we obtain
that
$$\int\limits_D K_{\mu_{G_m}}(z)\,dm(z)=
\int\limits_{\{z\in D: K_{\mu_{G_m}}(z)<t_0\}}
K_{\mu_{G_m}}(z)\,dm(z)+\int\limits_{\{z\in D:
K_{\mu_{G_m}}(z)\geqslant t_0\}} K_{\mu_{G_m}}(z)\,dm(z)\leqslant$$
$$\leqslant t_0\cdot m(D)+\frac{1}{C_1}\int\limits_D
\Phi(K_{\mu_{G_m}}(z))\,dm(z)\leqslant$$$$\leqslant t_0\cdot
m(D)+\frac{\sup\limits_{z\in D}(1+|z|^2)^2}{C_1}\int\limits_D
\Phi(K_{\mu_{G_m}}(z))\cdot\frac{1}{(1+|z|^2)^2}\,dm(z)\leqslant$$
\begin{equation}\label{eq7A}
\leqslant t_0\cdot m(D)+\frac{\sup\limits_{z\in
D}(1+|z|^2)^2}{C_1}\mathcal{M}(D)<\infty\,,
\end{equation}
because $\mathcal{M}(D)<\infty$ by the assumption of the theorem.

\medskip
\textbf{IV.} We prove that each map $G_m,$ $m=1,2,\ldots ,$ has a
continuous extension to $\partial D,$ in addition, the family of
extended maps $\overline{G}_m,$ $m=1,2,\ldots ,$ is equicontinuous
in $\overline{D}.$ Indeed, as proved in item~\textbf{III},
$K_{\mu_{G_m}}\in L^1(D).$ By~\cite[Theorem~3]{KPRS} (see
also~\cite[Theorem~3.1]{LSS}) each $G_m,$ $m=1,2,\ldots, $ is a ring
$Q$-homeomorphism in $\overline{D}$ for $Q=K_{\mu_{G_m}}(z),$ where
$\mu$ is defined in~(\ref{eq2C}), and $K_{\mu}$ my be calculated by
the formula~(\ref{eq1}). Thus, the desired conclusion is the
statement of Lemma~\ref{lem2}.

\medskip
\textbf{V.} Let us prove that the inverse homeomorphisms
$G^{\,-1}_m,$ $m=1,2,\ldots ,$ have a continuous extension
$\overline{G}^{\,-1}_m$ to $\partial {\Bbb D},$ and
$\{\overline{G}_m^{\,-1}\}_{m=1}^{\infty}$ is equicontinuous in
$\overline{\Bbb D}.$ Since by the item~\textbf{IV} mappings $G_m,$
$m=1,2,\ldots, $ a ring $K_{\mu_{G_m}}(z)$-homeomorphisms in $D,$
the corresponding inverse mappings $G^{\,-1}_m$
satisfy~(\ref{eq2*A}). Since $G^{\,-1}_m(0)=z_0$ for any
$m=1,2,\ldots ,$ the possibility of a continuous extension of
$G^{\,-1}_m$ to $\partial {\Bbb D},$ and the equicontinuity of
${\{\overline{G}_m^{\,-1}\}}_{m=1}^{\infty}$ on $\overline{\Bbb D}$
follows by Theorem~\ref{th4}.

\medskip
\textbf{VI.} Since, as proved above the family
${\{G_m\}}^{\infty}_{m=1}$ is equicontinuous in~$D,$ according to
Arzela-Ascoli criterion there exists an increasing subsequence of
numbers $m_k,$ $k=1,2,\ldots ,$ such that $G_{m_k}$ converges
locally uniformly in $D$ to some continuous mapping $G:D\rightarrow
\overline{{\Bbb C}}$ as $k\rightarrow\infty$  (see, e.g.,
\cite[Theorem~20.4]{Va}). By Lemma~\ref{lem1}, either $G$ is a
homeomorphism with values in ${\Bbb R}^n,$ or a constant
in~$\overline{{\Bbb R}^n}.$ Let us prove that the second case is
impossible. Let us apply the approach used in proof of the second
part of Theorem~21.9 in~\cite{Va}. Suppose the contrary: let
$G_{m_k}(x)\rightarrow c=const$ as $k\rightarrow\infty.$ Since
$G_{m_k}(z_0)=0$ for all $k=1,2,\ldots ,$ we have that $c=0.$ By
item~\textbf{V}, the family of mappings~$G^{\,-1}_m,$ $m=1,2,\ldots
,$ is equicontinuous in ${\Bbb D}.$ Then
$$h(z, G^{\,-1}_{m_k}(0))=h(G^{\,-1}_{m_k}(G_{m_k}(z)), G^{\,-1}_{m_k}(0))\rightarrow 0$$
as $k\rightarrow\infty,$ which is impossible because $z$ is an
arbitrary point of the domain $D.$ The obtained contradiction
refutes the assumption made above.

\medskip
\textbf{VII.} According to~\textbf{V}, the family of mappings
$\{\overline{G}_m^{\,-1}\}_{m=1}^{\infty}$ is equicontinuous
in~$\overline{\Bbb D}.$ By the Arzela-Ascoli criterion (see, e.g.,
\cite[Theorem~20.4]{Va}) we may consider that
$\overline{G}^{\,-1}_{m_k}(y),$ $k=1,2,\ldots, $ converges to some
mapping $\widetilde{F}:\overline{{\Bbb D}}\rightarrow \overline{D}$
as $k\rightarrow\infty$ uniformly in~$\overline{D}.$ Let us to prove
that $\widetilde{F}=\overline{G}^{\,-1}$ as $k\rightarrow\infty.$
For this purpose, we show that~$G(D)={\Bbb D}.$ Fix $y\in {\Bbb D}.$
Since $G_{m_k}(D)={\Bbb D}$ for every $k=1,2,\ldots, $ we obtain
that $G_{m_k}(x_k)=y$ for some $x_k\in D.$ Since $D$ is bounded, we
may consider that $x_k\rightarrow x_0\in \overline{D}$ as
$k\rightarrow\infty.$ By the triangle inequality and the
equicontinuity of ${\{\overline{G}_m\}}^{\infty}_{m=1}$ in
$\overline{D},$ see~\textbf{IV}, we obtain that
$$|\overline{G}
(x_0)-y|=|\overline{G}(x_0)-\overline{G}_{m_k}(x_k)|\leqslant
|\overline{G}(x_0)-\overline{G}_{m_k}(x_0)|+|\overline{G}_{m_k}(x_0)-\overline{G}_{m_k}(x_k)|\rightarrow
0$$
as $k\rightarrow\infty.$ Thus, $\overline{G}(x_0)=y.$ Observe that
$x_0\in D,$ because $G$ is a homeomorphism. Since $y\in {\Bbb D}$ is
arbitrary, the equality $G(D)={\Bbb D}$ is proved. In this case,
$G^{\,-1}_{m_k}\rightarrow G^{\,-1}$ locally uniformly in ${\Bbb D}$
as $k\rightarrow\infty$ (see, e.g., \cite[Lemma~3.1]{RSS}). Thus,
$\widetilde{F}(y)=G^{\,-1}(y)$ for every $y\in {\Bbb D}.$

Finally, since $\widetilde{F}(y)=G^{\,-1}(y)$ for any $y\in {\Bbb
D}$ and, in addition, $\widetilde{F}$ has a continuous extension on
$\partial{\Bbb D},$ due to the uniqueness of the limit at the
boundary points we obtain that
$\widetilde{F}(y)=\overline{G}^{\,-1}(y)$ for $y\in \overline{{\Bbb
D}}.$ Therefore, we have proved
that~$\overline{G}^{\,-1}_{m_k}\rightarrow \overline{G}^{\,-1}$
uniformly in~$\overline{\Bbb D}$ as $k\rightarrow\infty.$

\medskip
\textbf{VIII.} By~\textbf{VII,} for $y=e^{i\theta}\in
\partial {\Bbb D}$
\begin{equation}\label{eq4E}
{\rm
Re\,}F_{m_k}(e^{i\theta})=\varphi\left(\overline{G}^{\,-1}_{m_k}(e^{i\theta})\right)\rightarrow
\varphi\left(\overline{G}^{\,-1}(e^{i\theta})\right)
\end{equation}
as $k\rightarrow\infty$ uniformly on $\theta\in [0, 2\pi).$ Since by
the construction ${\rm Im\,}F_{m_k}(0)=0$ for any $k=1,2,\ldots,$ by
the Schwartz formula (see, e.g., \cite[section~8.III.3]{GK}) the
analytic function $F_{m_k}$ is uniquely restored by its real part,
namely,
\begin{equation}\label{eq4A}
F_{m_k}(y)=\frac{1}{2\pi i}\int\limits_{S(0,
1)}\varphi\left(\overline{G}^{\,-1}_{m_k}(t)\right)\frac{t+y}{t-y}\cdot\frac{dt}{t}\,.
\end{equation}
Set
\begin{equation}\label{eq5A}
F(y):=\frac{1}{2\pi i}\int\limits_{S(0,
1)}\varphi\left(\overline{G}^{\,-1}(t)\right)\frac{t+y}{t-y}\cdot\frac{dt}{t}\,.
\end{equation}
Let $K\subset {\Bbb D}$ be an arbitrary compact set, and let $y\in
K.$ By~(\ref{eq4A}) and~(\ref{eq5A}) we obtain that
\begin{equation}\label{eq11A}
|F_{m_k}(y)-F(y)|\leqslant \frac{1}{2\pi}\int\limits_{S(0,
1)}\bigl|\varphi(\overline{G}^{\,-1}_{m_k}(t))-\varphi(\overline{G}
^{\,-1}(t))\bigr|\left|\frac{t+y}{t-y}\right|\,|dt|\,.
\end{equation}
Since $K$ is compact, there is $0<R_0=R_0(K)<0$ such that $K\subset
B(0, R_0).$ By the triangle inequality $|t+y|\leqslant 1+R_0$ and
$|t-y|\geqslant |t|-|y|\geqslant 1-R_0$ for $y\in K$ and any $t\in
{\Bbb S}^1.$ Thus
\begin{equation}\label{eq12D}
\left|\frac{t+y}{t-y}\right|\leqslant \frac{1+R_0}{1-R_0}:=M=M(K)\,.
\end{equation}
Put $\varepsilon>0.$ By~(\ref{eq4E}), for a number
$\varepsilon^{\,\prime}:=\frac{\varepsilon}{M}$ there is
$N=N(\varepsilon, K)\in {\Bbb N}$ such that
$\bigl|\varphi\left(\overline{G}^{\,-1}_{m_k}(t)\right)-\varphi\left
(\overline{G}^{\,-1}(t)\right)\bigr|<\varepsilon^{\,\prime}$ for any
$k\geqslant N(\varepsilon)$ and $t\in {\Bbb S}^1.$ Now,
by~(\ref{eq11A})
\begin{equation}\label{eq13A}
|F_{m_k}(y)-F(y)|<\varepsilon \quad \forall\,\,k\geqslant N\,.
\end{equation}
It follows from~(\ref{eq13A}) that the sequence $F_{m_k}$ converhes
to $F$ as $k\rightarrow\infty$ locally uniformly in the unit disk.
In particular, we obtain that ${\rm Im\,}F(0)=0.$ Note that $F$ is
analytic function in ${\Bbb D}$ (see remarks made at the end of
item~8.III in~\cite{GK}), and
$${\rm Re}\,F(re^{i\psi})=\frac{1}{2\pi}\int\limits_0^{2\pi}
\varphi\left(\overline{G}^{\,-1}(e^{i\theta})\right)\frac{1-r^2}{1-2r\cos(\theta-\psi)+r^2}\,d\theta$$
for $z=re^{i\psi}.$
By~\cite[Theorem~2.10.III.3]{GK}
\begin{equation}\label{eq15A}
\lim\limits_{\zeta\rightarrow z}{\rm
Re}\,F(\zeta)=\varphi(\overline{G}^{\,-1}(z))\quad\forall\,\,z\in\partial
{\Bbb D}\,.
\end{equation}
Observe that $F$ either is a constant or open and discrete (see,
e.g., \cite[Ch.~V,  I.6 and II.5]{St}). Thus, $f_{m_k}=F_{m_k}\circ
G_{m_k}$ converges to $f=F\circ G$ locally uniformly as
$k\rightarrow\infty,$ where $f=F\circ G$ either is a constant or
open and discrete. Moreover, by~(\ref{eq15A})
$${\rm Re\,}f(z)={\rm Re\,}F(\overline{G}(z))=\varphi(\overline{G}^{\,-1}(\overline{G}(z)))=\varphi(z)\,.$$
\textbf{IX.} Since by~\textbf{VI} $G$ is a homeomorphism,
by~\cite[Lemma~1 and Theorem~1]{L} $G$  is a regular solution of the
equation~(\ref{eq2C}) for some function~$\mu:{\Bbb C}\rightarrow
{\Bbb D}.$ Since the set of points of the function $F,$ where its
Jacobian is zero, consist only of isolated points (see~\cite[Ch.~V,
5.II and 6.II]{St}), $f$ is regular solution of the Dirichlet
problem~(\ref{eq2C})--(\ref{eq1A}) whenever $F\not\equiv const.$
Note that the relation~(\ref{e3.3.1}) holds for the corresponding
function $K_{\mu}=K_{\mu_f}$ (see e.g.~\cite[Lemma~1]{L}).
Therefore, $f\in\frak{F}^{\mathcal{M}}_{\varphi, \Phi,
z_0}(D).$~$\Box$


\medskip
\medskip
{\bf \noindent Evgeny Sevost'yanov} \\
{\bf 1.} Zhytomyr Ivan Franko State University,  \\
40 Bol'shaya Berdichevskaya Str., 10 008  Zhytomyr, UKRAINE \\
{\bf 2.} Institute of Applied Mathematics and Mechanics\\
of NAS of Ukraine, \\
1 Dobrovol'skogo Str., 84 100 Slavyansk,  UKRAINE\\
esevostyanov2009@gmail.com

\medskip
{\bf \noindent Oleksandr Dovhopiatyi} \\
{\bf 1.} Zhytomyr Ivan Franko State University,  \\
40 Bol'shaya Berdichevskaya Str., 10 008  Zhytomyr, UKRAINE \\
alexdov1111111@gmail.com

\end{document}